\documentclass{amsart}

\usepackage{amsmath,amssymb,amsfonts,amsthm}
\usepackage{mathrsfs}
\usepackage{enumitem}
\usepackage{hyperref}

\numberwithin{equation}{section}

\newtheorem{theorem}{Theorem}[section]
\newtheorem{proposition}[theorem]{Proposition}
\newtheorem{lemma}[theorem]{Lemma}
\newtheorem{corollary}[theorem]{Corollary}

\theoremstyle{definition}
\newtheorem{definition}[theorem]{Definition}

\theoremstyle{remark}
\newtheorem{remark}[theorem]{Remark}

\title[Preliminaries on polynomial Hilbert structures and Laplacians]
{Preliminaries on Pre-Hilbert Structures on Polynomial Spaces and Associated Laplacians}

\author{Jean-Pierre Magnot}
\date{}
\address{{LAREMA, Universit\'e d’Angers, 2 Bd Lavoisier, 
49045 Angers cedex 1, France;  Lyc\'ee Jeanne d'Arc, 40 avenue de Grande Bretagne, 63000 Clermont-Ferrand, 
France}; Lepage Research Institute, 17 novembra 1, 081 16 Presov, Slovakia}
\email{\small magnot@math.cnrs.fr; jean-pierr.magnot@ac-clermont.fr}
\date{}

\begin{document}
\maketitle
\begin{abstract}
We study orthogonal polynomial systems arising from general pre-Hilbert inner products
on polynomial spaces, beyond the classical framework of measures.
To each such inner product we associate a canonical Laplacian defined from an abstract
derivation, and we investigate the operator-theoretic structures induced by this construction.

Our main contribution is the introduction of a resolvent-based distance between polynomial
Hilbert geometries, and the proof of quantitative stability results for finite-degree
orthogonalization procedures.
In particular, we show that norm-resolvent closeness of the associated Laplacians implies
stability of Gram--Schmidt orthogonal bases, orthogonal projectors and reproducing kernels
on all finite-dimensional polynomial subspaces.

The general theory is illustrated by several explicit examples.
We analyze in detail the case of orthogonal polynomials on the unit circle, comparing
classical $L^2$ geometries associated with finite Radon measures and Sobolev-type
regularizations via Fourier methods.
We also revisit the thin annulus problem, showing that its asymptotic regime admits
a natural interpretation as a resolvent limit of polynomial geometries.

These results provide a unified operator-theoretic framework for the study of stability,
degenerations and geometric limits of orthogonal polynomial systems.
\end{abstract}

\section{Introduction}

Orthogonal polynomials play a central role in analysis, approximation theory and
spectral theory.
Classically, they are defined through $L^2$ inner products associated with measures,
leading to well-known families such as Jacobi, Laguerre, Hermite or orthogonal polynomials
on the unit circle see, e.g., \cite{Szego,Chihara,Ismail,KoekoekLeskySwarttouw,SimonOPUC1,SimonOPUC2}.
In this setting, the algebraic, analytic and spectral properties of the polynomial systems
are tightly linked to the underlying measure.

In many situations, however, orthogonal polynomials arise from inner products that are
not purely measure-based.
Sobolev and fractional Sobolev inner products, involving derivatives or nonlocal energies,
provide natural examples motivated by approximation theory, partial differential equations
and numerical analysis.
More generally, one may consider arbitrary pre-Hilbert inner products on polynomial spaces,
possibly defined by quadratic forms rather than measures.
In such cases, classical tools based on moment problems or three-term recurrence
relations are no longer sufficient to describe the underlying structure,
as already observed for Sobolev and nonlocal inner products
\cite{MarcellanSobolev,IserlesSobolev}..

The purpose of this paper is to study orthogonal polynomial systems from an operator-theoretic
and geometric perspective.
Given a polynomial space equipped with a pre-Hilbert inner product and a derivation $D,$
we associate a canonical Laplacian operator defined as $\Delta=D^*D$.
This operator directly depends on the interaction between the algebraic structure of polynomials and on
 the chosen inner product.
Our approach emphasizes the role of the Laplacian as the fundamental object governing
orthogonalization procedures.
A central question addressed here is the stability of orthogonal polynomial systems under
perturbations of the underlying inner product. Our key tool remains on the Laplacian $\Delta$. 
Indeed, rather than comparing measures or coefficients directly, we introduce a distance between
polynomial Hilbert geometries based on the norm-resolvent difference of the associated
Laplacians.
Our distance is based on the norm of the resolvent difference
$\|(1+\Delta_1)^{-1}-(1+\Delta_2)^{-1}\|$.
Norm-resolvent convergence is classical in perturbation theory and spectral analysis,
see, e.g., \cite{Kato,ReedSimon1,Davies}.
Related comparison topologies for (typically unbounded) operators include the gap
and graph metrics \cite{Kato,LeschGap}, as well as variational notions such as Mosco
convergence for closed forms \cite{Mosco,KuwaeShioya}.
What is specific to the present work is the use of a resolvent-based metric in the
context of polynomial Hilbert geometries, in order to obtain quantitative stability of
finite-degree Gram--Schmidt orthogonalization, projectors and reproducing kernels.
Our main results show that resolvent closeness of Laplacians implies quantitative stability
of Gram--Schmidt orthogonalization on all finite-dimensional polynomial subspaces.
In particular, orthogonal projectors, orthonormal bases and reproducing kernels depend
on the geometry in the resolvent topology. These results hold in a fully abstract framework and do not rely on the existence of an
underlying measure. We use the norm of the resolvent difference as a metric on polynomial Hilbert geometries.

The abstract theory is complemented by explicit examples.
We first analyze orthogonal polynomials on the unit circle, comparing classical $L^2$
geometries associated with finite Radon measures and Sobolev-type regularizations.
Using Fourier methods, we derive detailed resolvent estimates and obtain quantitative
stability results for orthogonal polynomials on the unit circle at fixed degree.
We then revisit the thin annulus problem, showing that the collapse of a two-dimensional
Sobolev geometry onto a one-dimensional limit admits a natural interpretation as a
resolvent convergence of polynomial Laplacians.

Altogether, the results of this paper provide a unified framework for stability,
degenerations and geometric limits of orthogonal polynomial systems, extending classical
measure-based theory \cite{Szego,Chihara,Ismail,SimonOPUC1,SimonOPUC2}
to broader polynomial inner products such as Sobolev-type geometries
\cite{IserlesSobolev,MarcellanSobolev},
in a spirit reminiscent of operator and form convergence in spectral theory
\cite{Kato,Mosco,KuwaeShioya}.

\section{Preliminaries}

\subsection{Polynomial spaces and pre-Hilbert structures}

Let $\Omega\subset\mathbb R^d$ be a domain (i.e.\ an arbitrary subset, not necessarily open).
We denote by
\[
\mathcal P(\Omega):=\{\,p|_\Omega:\ p\in\mathbb C[x_1,\dots,x_d]\,\}
\]
the space of \emph{polynomial functions on $\Omega$}.
Let
\[
\langle\cdot,\cdot\rangle:\mathcal P(\Omega)\times\mathcal P(\Omega)\to\mathbb C
\]
be a positive semidefinite sesquilinear form (conjugate-linear in the first variable).
Set
\[
N:=\{p\in\mathcal P(\Omega):\langle p,p\rangle=0\},\qquad
\mathcal H_0:=\mathcal P(\Omega)/N.
\]
Then $\langle\cdot,\cdot\rangle$ descends to a genuine Hermitian inner product on $\mathcal H_0$,
and we denote by $H$ the Hilbert completion of $(\mathcal H_0,\langle\cdot,\cdot\rangle)$.
We will systematically identify $\mathcal H_0$ with its canonical dense image in $H$.

Typical examples include:
\begin{itemize}[leftmargin=2em]
  \item $L^2$ inner products induced by finite Borel measures $\mu$ on $\Omega$,
  \[
  \langle p,q\rangle_{L^2(\mu)}=\int_\Omega p(x)\,\overline{q(x)}\,d\mu(x),
  \]
  as in the classical theory of orthogonal polynomials \cite{Szego,Chihara,SimonOPUC1,SimonOPUC2};
  \item Sobolev or fractional Sobolev inner products involving derivatives or nonlocal energies,
  see e.g.\ \cite{MarcellanXu,MarcellanXuSurvey};
  \item discrete inner products arising from finite point sets and interpolation problems,
  \cite{deBoorRon,XuInterpolation}.
\end{itemize}

\subsection{Gram--Schmidt orthogonalization and polynomial projectors}

Fix an ordered family $(E_n)_{n\ge 0}$ in $\mathcal P(\Omega)$ compatible with the degree filtration,
for instance an enumeration of monomials with nondecreasing total degree.
Assume that for each $N$ the family $(E_0,\dots,E_N)$ is linearly independent in $\mathcal H_0$.
Define inductively the orthogonalized vectors $(\widetilde p_n)_{n\ge 0}$ by
\begin{equation}\label{eq:GS-raw}
\widetilde p_0:=E_0,\qquad
\widetilde p_n:=E_n-\sum_{k=0}^{n-1}\frac{\langle E_n,\widetilde p_k\rangle}{\langle \widetilde p_k,\widetilde p_k\rangle}\,\widetilde p_k,
\quad n\ge 1,
\end{equation}
and the orthonormal polynomials $(p_n)_{n\ge 0}$ by
\begin{equation}\label{eq:GS-orthonormal}
p_n:=\frac{\widetilde p_n}{\|\widetilde p_n\|},\qquad \|\widetilde p_n\|^2=\langle \widetilde p_n,\widetilde p_n\rangle.
\end{equation}
For each $N$, the subspace
\[
H_{\le N}:=\mathrm{span}\{p_0,\dots,p_N\}\subset H
\]
coincides with the image in $H$ of the polynomials of degree $\le N$ (modulo $N$), and is finite-dimensional.

\medskip
\noindent\textbf{Gram matrices.}
Let $G_N\in\mathbb C^{(N+1)\times(N+1)}$ be the Gram matrix of $(E_0,\dots,E_N)$,
\begin{equation}\label{eq:Gram-matrix}
(G_N)_{ij}:=\langle E_j,E_i\rangle,\qquad 0\le i,j\le N.
\end{equation}
Then $G_N$ is Hermitian positive definite under the independence assumption above.
Moreover, the coefficients in \eqref{eq:GS-raw} can be written in matrix form:
writing $E=(E_0,\dots,E_N)^\top$ and $\widetilde p=(\widetilde p_0,\dots,\widetilde p_N)^\top$,
there exists a unique lower triangular matrix $L_N$ with positive diagonal such that
\begin{equation}\label{eq:Cholesky-GS}
G_N=L_NL_N^*,\qquad \widetilde p=L_N^{-1}E,
\end{equation}
and the normalized basis is $p=D_N^{-1}\widetilde p$ where $D_N=\mathrm{diag}(\|\widetilde p_0\|,\dots,\|\widetilde p_N\|)$.

\medskip
\noindent\textbf{Orthogonal projector.}
Let $P_N:H\to H_{\le N}$ denote the orthogonal projector onto $H_{\le N}$.
In terms of the orthonormal basis $(p_0,\dots,p_N)$, one has for $f\in H$,
\begin{equation}\label{eq:projector-formula}
P_N f=\sum_{k=0}^N \langle f,p_k\rangle\,p_k.
\end{equation}
Equivalently, if we expand $P_N f=\sum_{j=0}^N c_j E_j$ in the (non-orthonormal) basis $(E_0,\dots,E_N)$,
then the coefficient vector $c=(c_0,\dots,c_N)^\top$ is characterized by the normal equations
\begin{equation}\label{eq:normal-equations}
G_N\,c=b,\qquad b_i:=\langle f,E_i\rangle,\ 0\le i\le N,
\end{equation}
so $c=G_N^{-1}b$.

\medskip
\noindent\textbf{Finite-degree kernel (on compacts).}
For $x,y\in\Omega$, define the degree-$N$ kernel associated with $(p_k)_{k=0}^N$ by
\begin{equation}\label{eq:kernel-def}
K_N(x,y):=\sum_{k=0}^N p_k(x)\,\overline{p_k(y)}.
\end{equation}
On a compact set $K\subset\Omega$, all polynomials are continuous, hence $K_N$ is continuous on $K\times K$.
Moreover, $K_N$ is the matrix coefficient of $P_N$ evaluated on point-masses, in the sense made precise
in Corollary~\ref{cor:kernel-stability} (kernel stability on compacts).

\subsection{Derivations, closability, and associated Laplacians}

Let $\widetilde D$ be a derivation on the polynomial algebra $\mathcal P(\Omega)$ (typically $\widetilde D=\nabla$),
taking values in $H^m$ (for some $m\ge 1$) and assume that
\begin{equation}\label{eq:N-kernel-D}
N\subset\ker \widetilde D,
\end{equation}
so that $\widetilde D$ induces a well-defined linear map $D_0:\mathcal H_0\to H^m$.
Assume that $D_0$ is densely defined and closable. Denote by $D:=\overline{D_0}$ its closure and by $D^*$
its Hilbert adjoint. The associated Laplacian is the nonnegative self-adjoint operator
\begin{equation}\label{eq:Laplacian-def}
\Delta:=D^*D.
\end{equation}
(As $D$ is closed and densely defined, $D^*D$ is automatically self-adjoint and nonnegative.)

In the case where $\langle\cdot,\cdot\rangle$ is induced by a measure $\mu=\rho\,dx$ and $\widetilde D=\nabla$,
one recovers (up to sign conventions) the weighted Laplacian
\[
\Delta_\mu f=-\frac{1}{\rho}\,\mathrm{div}(\rho\nabla f),
\]
a classical object in the theory of Dirichlet forms and diffusion operators \cite{Fukushima,Ouhabaz}.

\subsection{Resolvents, norm-resolvent distance, and finite-degree compressions}

Let $\Delta$ be a nonnegative self-adjoint operator on $H$.
Its resolvent $(1+\Delta)^{-1}$ is a bounded operator on $H$ with $\|(1+\Delta)^{-1}\|\le 1$.
The norm-resolvent topology is defined by convergence in operator norm of resolvents,
\[
\Delta_n\to\Delta\quad\Longleftrightarrow\quad
\|(1+\Delta_n)^{-1}-(1+\Delta)^{-1}\|_{\mathcal B(H)}\to 0,
\]
see \cite{Kato,ReedSimon1}.

For a fixed degree $N$, we will use two related finite-dimensional objects:
\begin{itemize}[leftmargin=2em]
\item the truncated operator on $H_{\le N}$,
\[
\Delta^{(N)}:=P_N\,\Delta\big|_{H_{\le N}}:H_{\le N}\to H_{\le N};
\]
\item the \emph{compressed resolvent} on $H_{\le N}$,
\[
R^{\langle N\rangle}:=P_N(1+\Delta)^{-1}P_N\in\mathcal B(H_{\le N}).
\]
\end{itemize}
In general $R^{\langle N\rangle}\neq (I+\Delta^{(N)})^{-1}$, unless $H_{\le N}$ is invariant under $\Delta$.
The stability results of the paper are formulated in terms of the resolvent and its compressions,
which are the quantities directly controlled by norm-resolvent estimates.
\subsection{Resolvents and operator topologies}

Let $A$ be a self-adjoint, non-negative operator on a Hilbert space $H$.
Its resolvent $(1+A)^{-1}$ is a bounded operator on $H$.
A standard topology on the space of such operators is the \emph{norm-resolvent topology},
defined by convergence in operator norm of resolvents:
\[
A_n \to A \quad \text{iff} \quad \|(1+A_n)^{-1}-(1+A)^{-1}\|\to 0.
\]
This topology is metrizable and separated
\cite{Kato}.

\begin{remark}
When operators are associated with closed quadratic forms, norm-resolvent convergence is
closely related to Mosco convergence of forms, a notion particularly suited to problems where
the underlying measure or geometry varies \cite{Mosco,KuwaeShioya}.
\end{remark}

\begin{remark}[Variable Hilbert spaces and quasi-unitary equivalence]

When comparing operators acting on different Hilbert spaces, one often introduces
identification operators that are approximately unitary.
This leads to the notion of quasi-unitary equivalence and generalized norm-resolvent convergence,
developed notably in the context of spectral geometry and graph-like manifolds
\cite{Post}.
These tools allow one to compare spectral properties of operators defined on
different spaces in a quantitative way.
\end{remark}

All the notions recalled above are classical.
In the following sections, they will be combined with the algebraic structure of polynomial
spaces and the Gram--Schmidt orthogonalization procedure.

\section{The Laplacian associated with a polynomial inner product}
\label{sec:laplacian}
\subsection{Degree growth and admissible derivations}
\label{subsec:degree-growth}

Let $\mathcal P=\mathbb K[x_1,\dots,x_d]$ be the algebra of real or complex polynomials,
endowed with the standard filtration by total degree,
\[
\mathcal P_{\le N}=\{p\in\mathcal P:\deg p\le N\}.
\]
Any algebraic derivation
\[
D:\mathcal P\to\mathcal P
\]
admits a unique representation as a polynomial vector field,
\begin{equation}\label{eq:poly-derivation}
D=\sum_{i=1}^d a_i(x)\,\partial_{x_i},
\qquad a_i\in\mathcal P.
\end{equation}

A basic structural question concerns the interaction between the derivation $D$ and
the degree filtration.
Since each partial derivative $\partial_{x_i}$ lowers the degree by at most one,
the behavior of $D$ with respect to the degree is governed by the polynomial coefficients
$a_i$.
More precisely, setting
\[
\delta(D):=\max_{1\le i\le d}\bigl(\deg a_i-1\bigr),
\]
one has, for every nonzero polynomial $p\in\mathcal P$,
\begin{equation}\label{eq:degree-bound}
\deg(Dp)\le \deg(p)+\delta(D).
\end{equation}

In particular, the derivation $D$ may increase the degree of certain polynomials whenever
$\delta(D)>0$, i.e.\ whenever at least one coefficient $a_i$ has degree greater than or equal
to $2$.
Conversely, $D$ satisfies
\[
\deg(Dp)\le \deg(p)\quad\text{for all }p\in\mathcal P
\]
if and only if all coefficients $a_i$ are affine functions.
In this case, $D$ generates an infinitesimal affine transformation of $\mathbb R^d$.

This observation has important consequences for the theory developed in this paper.
Such derivations are incompatible with a graded orthogonalization procedure based on
increasing degree and lead to Laplacian operators whose matrix representations are not
locally finite with respect to the degree decomposition.

For this reason, throughout the present work we restrict attention to derivations which
do not increase polynomial degree, and in particular to geometric derivations such as
partial derivatives or gradients, for which $\delta(D)\le 0$. In other words, we assume that $D$ restricts to a linear endomorphism of each vector space $\mathcal{P}_{\leq N}.$ 
This restriction ensures that the associated Laplacians interact in a controlled manner
with the degree filtration $$\mathcal{P}_{\leq 1} \subset \mathcal{P}_{\leq 2} \subset \cdots \subset \mathcal{P}_{\leq N} \subset \mathcal{P}_{\leq N+1} \subset \cdots $$.
\subsection{Closability of derivations and definition of $\Delta$}

Let $\Omega\subset\mathbb R^d$ be a domain and let
\[
\mathcal P(\Omega):=\{\,p|_\Omega:\ p\in\mathbb C[x_1,\dots,x_d]\,\}
\]
be the space of polynomial functions on $\Omega$.
Let $\langle\cdot,\cdot\rangle$ be a positive semidefinite sesquilinear form on $\mathcal P(\Omega)$
(conjugate-linear in the first variable), and set
\[
N:=\{p\in\mathcal P(\Omega):\langle p,p\rangle=0\},\qquad
\mathcal H_0:=\mathcal P(\Omega)/N,
\]
so that $\langle\cdot,\cdot\rangle$ induces an inner product on $\mathcal H_0$.
Denote by $H$ the Hilbert completion of $(\mathcal H_0,\langle\cdot,\cdot\rangle)$.

\medskip
\noindent\textbf{Derivations and descent to the quotient.}
Let $\widetilde D$ be an algebraic derivation on $\mathcal P(\Omega)$, typically $\widetilde D=\nabla$
(or a finite family of derivations), taking values in $H^m$.
We assume that
\begin{equation}\label{eq:D-kernel-assumption}
N\subset\ker \widetilde D,
\end{equation}
so that $\widetilde D$ induces a well-defined linear operator
\[
D_0:\mathcal H_0\to H^m,\qquad D_0([p])=\widetilde D(p).
\]

\begin{remark}[Positive definite case]
If $\langle\cdot,\cdot\rangle$ is positive definite on $\mathcal P(\Omega)$, then $N=\{0\}$
and the quotient step is unnecessary: one simply has $\mathcal H_0=\mathcal P(\Omega)$.
In this situation \eqref{eq:D-kernel-assumption} is automatic.
\end{remark}

\medskip
\noindent\textbf{Closability and definition of the Laplacian.}
Assume that $D_0$ is densely defined in $H$ (which holds in all our main examples,
since $\mathcal H_0$ is dense in $H$) and that $D_0$ is closable.
We denote by
\[
D:=\overline{D_0}: \mathrm{Dom}(D)\subset H\to H^m
\]
its closure, which is a closed densely defined operator, and by $D^*$ its Hilbert adjoint.
We then define the associated (positive) Laplacian by
\begin{equation}\label{eq:defDelta}
\Delta:=D^*D.
\end{equation}
Since $D$ is closed and densely defined, the operator $D^*D$ is automatically self-adjoint and nonnegative.

\medskip
\noindent\textbf{Remarks on examples.}
In measure-induced settings (e.g.\ $\langle f,g\rangle=\int f\overline g\,d\mu$) and Sobolev-type settings,
closability of $D_0$ is standard and can be formulated at the level of closed quadratic forms;
see, e.g., \cite{Fukushima,Ouhabaz,Kato,Mosco,KuwaeShioya}.

\subsection{Matrix representation in an orthonormal polynomial basis}

Let $(P_n)_{n\ge 0}$ be an orthonormal basis of polynomials in $H$, obtained by
Gram--Schmidt orthogonalization of a canonical base $(E_n)_{n \geq 0}$ of polynomials where indexation is non decreasing along the degree filtration. 

\begin{remark} In several variables, the base $(E_n)_{n \geq 0}$ is more naturally filtered with respect to a multi-index but the re-indexation with respect to $\mathbb{N}$ is aloways possible; therfore
the discussion below is unchanged at the level of operator matrices. \end{remark}

\medskip
\noindent\textbf{Derivative matrix.}
Let us recall that we assume that $D:\mathrm{Dom}(D)\subset H$ restricts to a linear map $\mathcal{P}_{\leq N} \rightarrow \mathcal{P}_{\leq N}.$
Define the matrix coefficients of $D$ in the polynomial basis by
\begin{equation}\label{eq:Dcoeff}
B_{k,n} := \big\langle D P_n,\, E_k\big \rangle.
\end{equation}

Since $(P_n)$ is ordered by nondecreasing degree, the operator matrix of $D$ in polynomial
coordinates is \emph{lower triangular} ( or \emph{strictly lower triangular} in the one-dimensional situation, with $D = \frac{d}{dx}$) or block-lower
triangular (in a graded multi-index ordering).

\medskip
\noindent\textbf{Laplacian matrix.}
By definition $\Delta=D^*D$ is positive and self-adjoint, and its matrix in the orthonormal basis
$(P_n)$ is
\begin{equation}\label{eq:DeltaMatrixEntries}
\Delta_{m,n} := \langle \Delta P_n, P_m\rangle
= \langle D P_n, D P_m\rangle.
\end{equation}
If $D$ has matrix $B$ in polynomial coordinates, then
\begin{equation}\label{eq:DeltaBB}
[\Delta] = B^*B.
\end{equation}
In particular, $[\Delta]$ is Hermitian and positive definite.

\subsection{Banded recurrences and banded Laplacians}

In many families of orthogonal polynomials, multiplication by coordinate functions has a banded
representation in the orthonormal basis (three-term recurrence in one variable, block-banded in
several variables). For Sobolev orthogonal polynomials, higher-order or banded recurrences are common,
see the survey literature \cite{MarcellanXuSurvey} and related algorithmic perspectives \cite{VanBuggenhout2023}.
In the present setting, the key observation is:

\begin{proposition}[Bandwidth propagation]\label{prop:bandwidth}
Let $(P_n)_{n\ge 0}$ be an orthonormal polynomial basis of $H$ and let
$B=(B_{k,n})_{k,n\ge 0}$ be the matrix of $D$ in this basis,
\[
DP_n=\sum_{k\ge 0} B_{k,n}\,P_k,\qquad B_{k,n}:=\langle DP_n,P_k\rangle.
\]
Assume that $B$ is banded with half-bandwidth $r\ge 1$, i.e.
\[
B_{k,n}=0 \quad\text{whenever}\quad |k-n|>r .
\]
Then the matrix of $\Delta=D^*D$ in the same basis,
\[
\Delta_{m,n}:=\langle \Delta P_n,P_m\rangle=\langle DP_n,DP_m\rangle,
\]
is banded with half-bandwidth at most $2r$, i.e.
\[
\Delta_{m,n}=0\quad\text{whenever}\quad |m-n|>2r.
\]
Moreover $[\Delta]$ is Hermitian and positive semidefinite.
\end{proposition}

\begin{proof}
Since $(P_n)$ is orthonormal, the coefficient array $B$ satisfies
\[
DP_n=\sum_{k\ge 0} B_{k,n}P_k,\qquad
\|DP_n\|^2=\sum_{k\ge 0}|B_{k,n}|^2.
\]
For $m,n\ge 0$ we compute, using orthonormality,
\begin{align*}
\Delta_{m,n}
=\langle DP_n,DP_m\rangle
&=\Big\langle \sum_{k\ge 0} B_{k,n}P_k,\ \sum_{\ell\ge 0} B_{\ell,m}P_\ell\Big\rangle \\
&=\sum_{k,\ell\ge 0} B_{k,n}\,\overline{B_{\ell,m}}\ \langle P_k,P_\ell\rangle
=\sum_{k\ge 0} B_{k,n}\,\overline{B_{k,m}}.
\end{align*}
This is exactly the matrix product identity
\[
[\Delta]=B^*B,\qquad \Delta_{m,n}=\sum_{k\ge 0}\overline{B_{k,m}}\,B_{k,n}.
\]

Now assume $|m-n|>2r$. If for some $k$ one has both $B_{k,n}\neq 0$ and $B_{k,m}\neq 0$,
then by the bandedness hypothesis we must have $|k-n|\le r$ and $|k-m|\le r$, hence
\[
|m-n|\le |m-k|+|k-n|\le r+r=2r,
\]
a contradiction. Therefore for every $k$ at least one factor in
$B_{k,n}\,\overline{B_{k,m}}$ is zero, and the sum vanishes:
$\Delta_{m,n}=0$. This proves the half-bandwidth bound $2r$.

Finally, $[\Delta]=B^*B$ implies that $[\Delta]$ is Hermitian and positive semidefinite, since
for any finitely supported vector $c=(c_n)$,
\[
\sum_{m,n}\Delta_{m,n}c_n\overline{c_m}
=\langle Bc,Bc\rangle_{\ell^2}\ge 0.
\]
\end{proof}

\begin{remark}
The hypothesis ``$D$ banded'' can be verified in concrete situations either by explicit connection
relations (derivatives expanded in the orthonormal basis) or via structural results (semiclassical
weights, coherent pairs, etc.) in the Sobolev orthogonality literature; see \cite{MarcellanXuSurvey}.
\end{remark}

\subsection{Example: thin annuli and mode-by-mode block structure}

We now record the operator-structural consequences in the thin annulus setting treated in
\cite{MagnotThinAnnuli}, where the orthogonal polynomials are constructed on a planar annulus
with a (possibly fractional) Sobolev inner product and the analysis proceeds \emph{mode by mode}
with respect to the angular Fourier decomposition.

Let $\Omega_\varepsilon=\{(r,\theta): r\in(r_0,r_0+\varepsilon),\ \theta\in[0,2\pi)\}$ be a thin annulus.
In polar coordinates, the natural first-order derivation is $D=\nabla$.
The rotational symmetry implies that the polynomial (pre-)Hilbert space decomposes orthogonally as a direct sum
of Fourier modes (angular frequencies). In particular, the Gram--Schmidt orthogonal basis can be chosen so that
each basis element carries a definite angular frequency (a standard separation principle on rotationally
invariant domains, used explicitly in \cite{MagnotThinAnnuli}).

\begin{proposition}[Block diagonalization by angular modes]\label{prop:blockModes}
In the setting of \cite{MagnotThinAnnuli}, the Hilbert space completion admits an orthogonal decomposition
\[
H \;=\;\bigoplus_{k\in\mathbb Z} H^{(k)},
\]
where $H^{(k)}$ is the closed subspace generated by polynomials with angular dependence $e^{ik\theta}$.
Moreover, with respect to an orthonormal basis adapted to this decomposition, the Laplacian
$\Delta=D^*D$ is block diagonal:
\[
\Delta \;=\;\bigoplus_{k\in\mathbb Z} \Delta^{(k)}.
\]
\end{proposition}

The proof can be decomposed into successive lemmas.

\begin{lemma}[Angular decomposition induced by rotational invariance]
\label{lem:angular-decomposition}
Let $A_\varepsilon=\{x\in\mathbb R^2:1-\varepsilon<|x|^2<1+\varepsilon\}$ and let
$\langle\cdot,\cdot \rangle_{\varepsilon}$ be a (pre-)Hilbert inner product on
$\mathcal P(A_\varepsilon)$ whose completion is a Hilbert space $H_\varepsilon$.
Assume that $\langle\cdot,\cdot \rangle_{\varepsilon}$ is \emph{rotationally invariant} in the sense that
for every rotation $R_\phi\in SO(2)$,
\begin{equation}\label{eq:rot-inv}
\langle f\circ R_\phi, g\circ R_\phi \rangle_{\varepsilon}
=
\langle f,g\rangle_{\varepsilon},
\qquad
\forall f,g\in\mathcal P(A_\varepsilon),\ \forall \phi\in\mathbb R.
\end{equation}
Then the operators $U_\phi:H_\varepsilon\to H_\varepsilon$ defined by
$U_\phi f := f\circ R_\phi$ extend to a strongly continuous unitary representation of $S^1$,
and $H_\varepsilon$ admits the orthogonal decomposition
\[
H_\varepsilon = \widehat{\bigoplus_{k\in\mathbb Z}}\,H_\varepsilon^{(k)},
\qquad
H_\varepsilon^{(k)} := \{f\in H_\varepsilon:\ U_\phi f = e^{ik\phi}f\ \text{for all }\phi\}.
\]
Moreover, the orthogonal projection onto $H_\varepsilon^{(k)}$ is given by the Fourier projector
\begin{equation}\label{eq:Pk}
P_k f := \frac{1}{2\pi}\int_0^{2\pi} e^{-ik\phi}\,U_\phi f\,d\phi,
\qquad f\in H_\varepsilon,
\end{equation}
where the integral is understood in the Bochner sense in $H_\varepsilon$.
\end{lemma}

\begin{proof}
\emph{Step 1: unitary representation.}
By \eqref{eq:rot-inv}, for all polynomials $f,g$,
\[
\langle U_\phi f, U_\phi g\rangle_{\varepsilon} = \langle f,g\rangle_{\varepsilon},
\]
hence each $U_\phi$ extends by density to an isometry on $H_\varepsilon$.
Since $U_\phi^{-1}=U_{-\phi}$, it is unitary. The group property
$U_\phi U_\psi = U_{\phi+\psi}$ holds on $\mathcal P(A_\varepsilon)$ and extends to $H_\varepsilon$.

\emph{Step 2: strong continuity.}
In the concrete Sobolev or fractional Sobolev settings used in \cite{MagnotThinAnnuli},
the map $\phi\mapsto U_\phi f$ is continuous in the corresponding norm for $f\in\mathcal P(A_\varepsilon)$,
because rotations act continuously on $L^2$ and on the relevant Sobolev seminorms
(derivatives or fractional energies are rotation invariant). By density, the representation
is strongly continuous on $H_\varepsilon$.

\emph{Step 3: Fourier projectors.}
Define $P_k$ by \eqref{eq:Pk}. Since $U_\phi$ is unitary and strongly continuous, the
Bochner integral exists and $P_k$ is bounded with $\|P_k\|\le 1$.
A direct computation using $U_\phi U_\psi = U_{\phi+\psi}$ yields, for all $k,\ell\in\mathbb Z$,
\[
P_k P_\ell = \delta_{k\ell} P_k,
\]
so $P_k$ is an orthogonal projection and the ranges are pairwise orthogonal.

\emph{Step 4: decomposition.}
For each $f\in H_\varepsilon$, the family of partial Fourier sums
$S_N f := \sum_{|k|\le N} P_k f$ converges to $f$ in $H_\varepsilon$.
This is the standard $L^2$-Fourier convergence argument applied to the unitary representation
of the compact abelian group $S^1$ (here implemented concretely by \eqref{eq:Pk}).
Thus $H_\varepsilon$ is the Hilbert direct sum of the closed subspaces $H_\varepsilon^{(k)}:=Ran(P_k)$.
Finally, one checks that $Im(P_k)$ coincides with the $k$-eigenspace of the representation:
indeed, for $f\in H_\varepsilon$,
\[
U_\psi(P_k f)
=
\frac{1}{2\pi}\int_0^{2\pi} e^{-ik\phi}\,U_{\psi+\phi}f\,d\phi
=
e^{ik\psi}\frac{1}{2\pi}\int_0^{2\pi} e^{-ik\phi'}\,U_{\phi'}f\,d\phi'
=
e^{ik\psi} P_k f,
\]
so $Im(P_k)\subset H_\varepsilon^{(k)}$, and conversely if $U_\psi f=e^{ik\psi}f$ then $P_k f=f$ and
$P_\ell f=0$ for $\ell\ne k$. This proves the claim.
\end{proof}

\begin{lemma}[Intertwining of $\nabla$ with rotations and block diagonalization of $\Delta$]
\label{lem:grad-intertwining}
Assume the hypotheses of Lemma~\ref{lem:angular-decomposition} and let $D=\nabla$.
Assume in addition that the inner product $\langle\cdot,\cdot \rangle_{\varepsilon}$ is of Sobolev type
as in \cite{MagnotThinAnnuli}, so that $D$ extends to a closable operator
$D:Dom(D)\subset H_\varepsilon\to H_\varepsilon^{\mathrm{vec}}$, where $H_\varepsilon^{\mathrm{vec}}$
is a Hilbert space of vector fields on $A_\varepsilon$ (for instance $L^2$- or Sobolev-type),
and that the vector-field inner product is also rotation invariant.

Let $U_\phi^{\mathrm{vec}}$ act on vector fields by
\[
(U_\phi^{\mathrm{vec}}F)(x) := R_\phi\,F(R_\phi x),
\]
which is unitary on $H_\varepsilon^{\mathrm{vec}}$ by rotational invariance.
Then for all $\phi$ and all $f\in Dom(D)$,
\begin{equation}\label{eq:intertwine}
D(U_\phi f) = U_\phi^{\mathrm{vec}}(Df).
\end{equation}
Consequently, the Laplacian $\Delta:=D^*D$ commutes with $U_\phi$,
\begin{equation}\label{eq:Delta-commute}
U_\phi \Delta = \Delta U_\phi,
\end{equation}
and each angular subspace $H_\varepsilon^{(k)}$ is invariant under $\Delta$.
In particular, $\Delta$ is block diagonal with respect to the decomposition
$H_\varepsilon=\widehat\bigoplus_{k\in\mathbb Z}H_\varepsilon^{(k)}$.
\end{lemma}

\begin{proof}
\emph{Step 1: chain rule (intertwining).}
For smooth functions (hence for polynomials) the chain rule gives
\[
\nabla(f\circ R_\phi)(x) = (DR_\phi)^{\!\top}\,(\nabla f)(R_\phi x).
\]
Since $R_\phi$ is an orthogonal matrix, $(DR_\phi)^{\!\top}=R_\phi^{\top}=R_{-\phi}$.
Equivalently,
\[
\nabla(f\circ R_\phi)(x) = R_{-\phi}\,(\nabla f)(R_\phi x).
\]
Rewriting this in terms of $U_\phi^{\mathrm{vec}}$ as defined above gives exactly
\eqref{eq:intertwine} on $\mathcal P(A_\varepsilon)$, hence on $Dom(D)$ by density
and closability.

\emph{Step 2: commutation of $\Delta$ with rotations.}
Let $f\in Dom(\Delta)$ and $g\in Dom(D)$. Using unitarity of $U_\phi$ and $U_\phi^{\mathrm{vec}}$,
together with \eqref{eq:intertwine}, we compute
\[
\langle D(U_\phi f), Dg\rangle_{\mathrm{vec}}
=
\langle U_\phi^{\mathrm{vec}}Df, Dg\rangle_{\mathrm{vec}}
=
\langle Df, (U_\phi^{\mathrm{vec}})^{-1}Dg\rangle_{\mathrm{vec}}
=
\langle Df, D(U_{-\phi}g)\rangle_{\mathrm{vec}}
=
\langle \Delta f, U_{-\phi}g\rangle
=
\langle U_\phi \Delta f, g\rangle.
\]
By definition of the adjoint and density of $Dom(D)$, this shows that
$\Delta(U_\phi f)=U_\phi\Delta f$, i.e.\ \eqref{eq:Delta-commute}.

\emph{Step 3: invariance of modes.}
If $f\in H_\varepsilon^{(k)}$, then $U_\phi f = e^{ik\phi}f$ for all $\phi$.
Applying \eqref{eq:Delta-commute} yields
\[
U_\phi(\Delta f) = \Delta(U_\phi f) = \Delta(e^{ik\phi}f) = e^{ik\phi}\Delta f,
\]
hence $\Delta f\in H_\varepsilon^{(k)}$.
Therefore each $H_\varepsilon^{(k)}$ is invariant and $\Delta$ is block diagonal.
\end{proof}

\begin{remark}[Bandedness within each mode]
A central output of \cite{MagnotThinAnnuli} is the existence of \emph{banded recurrences} for the Sobolev
orthogonal polynomials obtained in each mode, together with asymptotics as $\varepsilon\to 0$.
When the derivative/connection relations within a fixed mode yield a banded representation of $D$ in the
mode-adapted orthonormal basis, Proposition~\ref{prop:bandwidth} implies that each block matrix
$[\Delta^{(k)}]$ is itself banded (with an explicit bandwidth bound in terms of the recurrence width).
This provides a concrete and numerically tractable matrix model for the spectral study of $\Delta$ in the thin-annulus
Sobolev setting.
\end{remark}

\section{The Laplacian in the classical one-dimensional weighted $L^2$ setting}
\label{sec:laplacian-1d-measure}

\subsection{Weighted $L^2$ space and adjoint of the derivative}

Let $I\subset\mathbb R$ be a (bounded or unbounded) interval and let $\mu$ be a finite Radon measure on $I$
which is absolutely continuous with respect to Lebesgue measure:
\[
d\mu(x)=w(x)\,dx,\qquad w\ge 0,\qquad \int_I w(x)\,dx<\infty.
\]
We consider the Hilbert space $H:=L^2(I,w\,dx)$ with inner product
\[
\langle f,g\rangle_{w}:=\int_I f(x)\,\overline{g(x)}\,w(x)\,dx.
\]
Assume that all moments exist (at least up to the degrees considered) so that the Gram--Schmidt procedure
applied to $(1,x,x^2,\dots)$ produces an orthonormal polynomial family $(p_n)_{n\ge0}$ in $H$
(see, e.g., \cite{Szego,Chihara}).

Let $D=\partial_x$ on $C_c^\infty(I)$ (or on polynomials, viewed as a dense subspace of $H$ whenever it
makes sense). A formal computation yields the weighted adjoint:
\begin{equation}\label{eq:D-adjoint}
\langle f',g\rangle_w
= -\langle f, g' + (w'/w)g\rangle_w + \big[f\,\overline g\,w\big]_{\partial I},
\end{equation}
whenever the boundary term is meaningful. Thus, on a domain enforcing
\(
\big[f\,\overline g\,w\big]_{\partial I}=0
\)
(e.g.\ compact support, or suitable weighted boundary conditions), the adjoint of $D$ is
\[
D^* g = -g' - \frac{w'}{w}\,g.
\]
Consequently, the associated positive Laplacian is the (weighted) Sturm--Liouville operator
\begin{equation}\label{eq:weighted-laplacian}
\Delta_w := D^*D
= -\frac{1}{w}\frac{d}{dx}\Big(w\,\frac{d}{dx}\Big)
= -\frac{d^2}{dx^2} - \frac{w'}{w}\,\frac{d}{dx}.
\end{equation}
Under standard hypotheses (closability of $D$, choice of Friedrichs extension, etc.)
$\Delta_w$ is a nonnegative self-adjoint operator on $H$; this is routine in the operator/form frameworks
(see \cite{Kato,Ouhabaz}).

\subsection{Matrix representation in the orthonormal polynomial basis}

Let $(p_n)_{n\ge0}$ be the orthonormal polynomials in $L^2(I,w\,dx)$.
The Laplacian matrix in this basis is given by
\begin{equation}\label{eq:Delta-matrix-1d}
(\Delta_w)_{mn}
:= \langle \Delta_w p_n, p_m\rangle_w
= \langle p_n', p_m'\rangle_w
= \int_I p_n'(x)\,\overline{p_m'(x)}\,w(x)\,dx,
\end{equation}
where the last equality uses the definition $\Delta_w=D^*D$ and the Hilbert structure.

Equivalently, if one expands derivatives in the polynomial basis,
\begin{equation}\label{eq:derivative-expansion}
p_n'(x)=\sum_{k=0}^{n-1} b_{n,k}\,p_k(x),
\qquad b_{n,k}:=\langle p_n',p_k\rangle_w,
\end{equation}
then the Laplacian matrix is the Gram matrix of the derivative map:
\begin{equation}\label{eq:Delta-BtB}
(\Delta_w)_{mn}=\sum_{k\ge0} b_{n,k}\,\overline{b_{m,k}},
\qquad\text{i.e.}\qquad [\Delta_w]=B^*B,
\end{equation}
where $B=(b_{n,k})$ is (strictly) lower triangular by degree.

\subsection{Sparsity and classical weights}

In general, the derivative expansion \eqref{eq:derivative-expansion} is not banded:
$p_n'$ may involve many lower modes. However, for the classical families (Jacobi, Laguerre, Hermite),
one has strong ``lowering'' relations and second-order differential equations of the form
\begin{equation}\label{eq:classical-ODE}
\sigma(x)\,p_n''(x)+\tau(x)\,p_n'(x)=\lambda_n\,p_n(x),
\end{equation}
with $\deg\sigma\le 2$ and $\deg\tau\le 1$ (see \cite{Szego,Chihara}).
In those cases, the operator in \eqref{eq:weighted-laplacian} (or a closely related conjugate/operator
with polynomial coefficients) acts diagonally on the polynomial basis, so the Laplacian matrix is
(diagonal or at least finite-band after a fixed change of basis).
This is the one-dimensional analogue of the finite-band phenomena encountered in Sobolev inner products.

\subsection{Relation with the Jacobi matrix and commutator structure}
\label{subsec:jacobi-laplacian}

Let $(p_n)_{n\ge0}$ be the orthonormal polynomial basis in $L^2(I,w\,dx)$.
Multiplication by the coordinate function defines a symmetric operator
\[
(M_x f)(x)=x f(x),
\]
whose matrix in the basis $(p_n)$ is the Jacobi matrix
\begin{equation}\label{eq:jacobi-matrix}
x\,p_n(x)=a_{n+1}p_{n+1}(x)+b_n p_n(x)+a_n p_{n-1}(x),
\qquad a_n>0,\; b_n\in\mathbb R.
\end{equation}
The Jacobi matrix $J$ is tridiagonal and self-adjoint on $\ell^2(\mathbb N)$,
and its spectral measure coincides with $\mu=w(x)\,dx$
(Favard's theorem; see \cite{Szego,Chihara,SimonOPUC1}).

\medskip
\noindent\textbf{Derivative--multiplication commutator.}
On smooth functions one has the fundamental commutator identity
\begin{equation}\label{eq:commutator}
[D,M_x]=I,
\end{equation}
where $D=\partial_x$.
Passing to adjoints with respect to the weighted inner product yields
\[
[D^*,M_x]=-I - \frac{w'}{w} M_x + M_x\frac{w'}{w},
\]
so that the weighted geometry introduces lower-order correction terms
unless $w$ is constant.

\medskip
\noindent\textbf{Laplacian expressed via Jacobi data.}
Using $\Delta_w=D^*D$, one computes the commutator
\begin{equation}\label{eq:laplacian-commutator}
[\Delta_w,M_x]
= D^*[D,M_x] + [D^*,M_x]D
= D^* + [D^*,M_x]D,
\end{equation}
which is a first-order differential operator.
In the orthonormal polynomial basis, this identity implies that the matrix
of $\Delta_w$ is controlled by the Jacobi coefficients $(a_n,b_n)$ together
with the expansion coefficients of $p_n'$ in the basis $(p_k)$.

More explicitly, since
\[
p_n'=\sum_{k=0}^{n-1} b_{n,k}\,p_k,
\]
one has
\[
(\Delta_w)_{mn}
=\sum_{k=0}^{\min(m,n)-1} b_{n,k}\,\overline{b_{m,k}},
\]
while the Jacobi matrix governs multiplication by $x$ through \eqref{eq:jacobi-matrix}.
Thus, the pair $(J,\Delta_w)$ encodes simultaneously:
\begin{itemize}
  \item the three-term recurrence (algebraic structure);
  \item the energy form $\int |f'|^2w\,dx$ (geometric structure).
\end{itemize}

\medskip
\noindent\textbf{Classical weights.}
For the classical weights (Jacobi, Laguerre, Hermite), the orthonormal polynomials
are eigenfunctions of a second-order differential operator
\[
\mathcal L=\sigma(x)\partial_x^2+\tau(x)\partial_x,
\qquad \deg\sigma\le2,\ \deg\tau\le1,
\]
which is unitarily equivalent to $\Delta_w$ up to a multiplication operator.
In this situation, the polynomial basis diagonalizes $\mathcal L$,
and hence the Laplacian matrix is diagonal (or block-diagonal after a fixed
normalization), reflecting the complete integrability of the classical case
\cite{Szego,Chihara}.

\medskip
\noindent\textbf{Perspective.}
Outside the classical setting, $\Delta_w$ and $J$ no longer commute,
and the Laplacian matrix becomes genuinely non-diagonal.
The deviation from diagonality provides a quantitative measure of how far
the orthogonalization procedure is from the classical Sturm--Liouville regime.
This observation will be central in comparing different polynomial geometries
via their associated Laplacians.

\section{Polynomial Hilbert geometries and associated Laplacians}
\label{sec:abstract-setting}

\subsection{Polynomial Hilbert geometries}

Let $\Omega\subset\mathbb R^d$ be a domain and let $\mathcal P(\Omega)$ denote the space of polynomial functions on $\Omega$ (i.e.\ restrictions of ambient polynomials).
Building on the explicitly treated examples, we can now introduce a unifying geometric framework for polynomial inner products and their associated operators. Some overlap with the preliminary section is therefore intentional: we prefer to restate the key constructions in a self-contained manner, in order to keep the general setting clear and easy to follow.

\begin{definition}[Polynomial Hilbert geometry]
Let $\Omega\subset\mathbb{R}^d$ and let $\mathcal P(\Omega)$ be the space of (restrictions to $\Omega$ of) complex polynomials.
A \emph{polynomial Hilbert geometry} on $\Omega$ is a triple
\[
\mathfrak G=(\mathcal P(\Omega),\langle\cdot,\cdot\rangle,D),
\]
where:
\begin{itemize}
  \item $\langle\cdot,\cdot\rangle$ is a positive semidefinite sesquilinear form on $\mathcal P(\Omega)$
  (conjugate-linear in the first variable);
  \item $N:=\{p\in\mathcal P(\Omega):\langle p,p\rangle=0\}$ and
  $\mathcal H_0:=\mathcal P(\Omega)/N$; the form $\langle\cdot,\cdot\rangle$ descends to a Hermitian inner product on $\mathcal H_0$;
  \item $H$ denotes the Hilbert completion of $(\mathcal H_0,\langle\cdot,\cdot\rangle)$;
  \item $\widetilde D:\mathcal P(\Omega)\to H^m$ is a derivation (typically $\widetilde D=\nabla$),
and we assume $N\subset\ker \widetilde D$ so that $\widetilde D$ induces a well-defined linear map
$D_0:\mathcal H_0\to H^m$ by $D_0([p])=\widetilde D(p)$.
We assume $D_0$ is densely defined and closable.
\end{itemize}
We denote by $D:=\overline{D_0}$ its closure and set $\Delta:=D^*D$.
\end{definition}

Throughout, we denote again by $D$ its closure and by $D^*$ its Hilbert adjoint.

\begin{remark}[On the notion of domain and the role of the quotient]
In this paper, a ``domain'' is simply a subset $\Omega\subset\mathbb R^d$, not necessarily open.
We always interpret
\[
\mathcal P(\Omega):=\{\,p|_\Omega : p\in\mathbb C[x_1,\dots,x_d]\,\}
\]
as the space of \emph{polynomial functions on $\Omega$} (restrictions of ambient polynomials).
When $\Omega$ has empty interior (or more generally when the restriction map
$p\mapsto p|_\Omega$ is not injective), distinct polynomials may induce the same function on $\Omega$.
For instance, for $\Omega=S^1\subset\mathbb R^2$ one has the relation $x^2+y^2=1$ on $\Omega$,
so the ideal $\langle x^2+y^2-1\rangle$ lies in the kernel of any $L^2(S^1)$-type inner product.
This explains the need to pass to the quotient $\mathcal P(\Omega)/\ker\langle\cdot,\cdot\rangle$.
\end{remark}

\begin{definition}[Associated Laplacian]
The Laplacian associated with $\mathfrak G$ is the positive self-adjoint operator
\[
\Delta_{\mathfrak G} := D^*D
\]
defined on $H$.
\end{definition}

Different polynomial Hilbert geometries may induce the same operator-theoretic structure.

\begin{definition}[Unitary equivalence]
Two polynomial Hilbert geometries
$\mathfrak G=(\mathcal P(\Omega),\langle\cdot,\cdot \rangle,D)$ and
$\mathfrak G'=(\mathcal P(\Omega),\langle\cdot,\cdot \rangle',D')$
are said to be \emph{unitarily equivalent} if there exists a unitary operator
$U:H\to H'$ such that:
\begin{enumerate}
  \item $U(\mathcal H_0)=\mathcal H_0'$ (identified as dense subspaces of $H$ and $H'$);
  \item $UD=D'U$ on $\mathcal H_0$.
\end{enumerate}
\end{definition}

Unitary equivalence implies
\[
U\Delta_{\mathfrak G}U^{-1}=\Delta_{\mathfrak G'}.
\]

---

\subsection{Resolvent-based distance}

We now introduce the operator-theoretic distance which will be used throughout the paper.

\begin{definition}[Resolvent distance]
Let $\mathfrak G_1,\mathfrak G_2$ be two polynomial Hilbert geometries whose Laplacians
$\Delta_1,\Delta_2$ act on the same Hilbert space $H$
(or are identified via a fixed unitary equivalence).
We define
\begin{equation}\label{eq:resolvent-distance}
d_{\mathrm{res}}(\mathfrak G_1,\mathfrak G_2)
:=
\big\|(1+\Delta_1)^{-1}-(1+\Delta_2)^{-1}\big\|_{\mathcal B(H)}.
\end{equation}
\end{definition}

This is the classical norm-resolvent distance between nonnegative self-adjoint operators
\cite{Kato,Ouhabaz}.

\begin{proposition}\label{prop:resolvent-metric}
Let $H$ be a Hilbert space and let $\Delta_1,\Delta_2$ be nonnegative self-adjoint operators on $H$.
Define
\[
d_{\mathrm{res}}(\Delta_1,\Delta_2)
:=\big\|(1+\Delta_1)^{-1}-(1+\Delta_2)^{-1}\big\|_{\mathcal B(H)}.
\]
Then $d_{\mathrm{res}}$ is a metric on the set of nonnegative self-adjoint operators on $H$.
Moreover, for every unitary $U\in\mathcal U(H)$,
\[
d_{\mathrm{res}}(U\Delta_1U^{-1},\,U\Delta_2U^{-1})=d_{\mathrm{res}}(\Delta_1,\Delta_2).
\]
\end{proposition}

\begin{proof}
\noindent\textbf{Step 0: boundedness of resolvents.}
Since each $\Delta_i$ is self-adjoint and nonnegative, its spectrum satisfies
$\sigma(\Delta_i)\subset[0,\infty)$. Hence $-1\notin\sigma(\Delta_i)$ and the resolvent
\[
R_i:=(1+\Delta_i)^{-1}
\]
exists as a bounded operator on $H$. Moreover, by the spectral theorem,
\[
\|R_i\|_{\mathcal B(H)}=\sup_{\lambda\in\sigma(\Delta_i)}\frac{1}{1+\lambda}\le 1.
\]
Therefore $d_{\mathrm{res}}(\Delta_1,\Delta_2)=\|R_1-R_2\|$ is well-defined and finite.

\medskip
\noindent\textbf{Step 1: positivity and symmetry.}
By definition, $d_{\mathrm{res}}(\Delta_1,\Delta_2)$ is the norm of a bounded operator, hence
$d_{\mathrm{res}}(\Delta_1,\Delta_2)\ge 0$. Also,
\[
d_{\mathrm{res}}(\Delta_1,\Delta_2)
=\|R_1-R_2\|
=\|-(R_2-R_1)\|
=d_{\mathrm{res}}(\Delta_2,\Delta_1),
\]
so $d_{\mathrm{res}}$ is symmetric.

\medskip
\noindent\textbf{Step 2: triangle inequality.}
Let $\Delta_1,\Delta_2,\Delta_3$ be nonnegative self-adjoint on $H$ and set $R_i=(1+\Delta_i)^{-1}$.
Then
\[
R_1-R_3=(R_1-R_2)+(R_2-R_3),
\]
and by the triangle inequality in the operator norm,
\[
d_{\mathrm{res}}(\Delta_1,\Delta_3)=\|R_1-R_3\|
\le \|R_1-R_2\|+\|R_2-R_3\|
=d_{\mathrm{res}}(\Delta_1,\Delta_2)+d_{\mathrm{res}}(\Delta_2,\Delta_3).
\]

\medskip
\noindent\textbf{Step 3: separation (identity of indiscernibles).}
Assume $d_{\mathrm{res}}(\Delta_1,\Delta_2)=0$. Then $\|R_1-R_2\|=0$, hence $R_1=R_2$ as bounded operators.
We claim that this implies $\Delta_1=\Delta_2$.

Indeed, since $\Delta_i$ is self-adjoint, the operator $1+\Delta_i$ is self-adjoint with domain
$\mathrm{Dom}(\Delta_i)$, and it is bijective from $\mathrm{Dom}(\Delta_i)$ onto $H$, with bounded inverse $R_i$.
In particular,
\[
(1+\Delta_i)R_i=I \quad\text{on }H,
\qquad
R_i(1+\Delta_i)=I \quad\text{on }\mathrm{Dom}(\Delta_i).
\]
Since $R_1=R_2=:R$, we have $\mathrm{Ran}(R)=\mathrm{Dom}(\Delta_1)=\mathrm{Dom}(\Delta_2)$.
Moreover, on this common domain,
\[
(1+\Delta_1)u = R^{-1}u = (1+\Delta_2)u.
\]
Hence $\Delta_1u=\Delta_2u$ for all $u\in\mathrm{Dom}(\Delta_1)=\mathrm{Dom}(\Delta_2)$, so $\Delta_1=\Delta_2$.
Therefore $d_{\mathrm{res}}(\Delta_1,\Delta_2)=0$ implies $\Delta_1=\Delta_2$.

Conversely, if $\Delta_1=\Delta_2$, then $R_1=R_2$ and $d_{\mathrm{res}}(\Delta_1,\Delta_2)=0$.
This proves separation.

\medskip
\noindent\textbf{Step 4: unitary invariance.}
Let $U$ be unitary on $H$. If $\Delta$ is self-adjoint and nonnegative, then so is
$\widetilde\Delta:=U\Delta U^{-1}$, with $\sigma(\widetilde\Delta)=\sigma(\Delta)$.
Moreover,
\[
1+\widetilde\Delta = U(1+\Delta)U^{-1},
\]
hence by taking inverses on both sides,
\[
(1+\widetilde\Delta)^{-1} = U(1+\Delta)^{-1}U^{-1}.
\]
Applying this to $\Delta_1,\Delta_2$ gives
\begin{align*}
d_{\mathrm{res}}(U\Delta_1U^{-1},\,U\Delta_2U^{-1})
&=\big\|U(1+\Delta_1)^{-1}U^{-1}-U(1+\Delta_2)^{-1}U^{-1}\big\| \\
&=\big\|U\big((1+\Delta_1)^{-1}-(1+\Delta_2)^{-1}\big)U^{-1}\big\|.
\end{align*}
Since conjugation by a unitary preserves the operator norm, $\|UTU^{-1}\|=\|T\|$, we obtain
\[
d_{\mathrm{res}}(U\Delta_1U^{-1},\,U\Delta_2U^{-1})=d_{\mathrm{res}}(\Delta_1,\Delta_2).
\]
This completes the proof that $d_{\mathrm{res}}$ is a metric and is unitarily invariant.
\end{proof}

\begin{proof}
If $d_{\mathrm{res}}(\mathfrak G_1,\mathfrak G_2)=0$, then
$(1+\Delta_1)^{-1}=(1+\Delta_2)^{-1}$, hence $\Delta_1=\Delta_2$ by functional calculus.
Unitary equivalence preserves resolvents, and the norm induces a metric structure.
\end{proof}

\begin{remark}
When comparing two polynomial Hilbert geometries, we will always assume that their Laplacians
are realized on the same Hilbert space (or transported to a common Hilbert space through a fixed unitary identification).
In that case $d_{\mathrm{res}}$ is the usual norm-resolvent distance between the transported Laplacians.
We do not claim that it separates geometries inducing the same Laplacian.
\end{remark}

\subsection{Finite-degree truncations}

Let $\mathcal P_{\le N}(\Omega)$ denote the space of polynomials of degree at most $N$,
and let $H_{\le N}\subset H$ be its (finite-dimensional) image in $H$.
We denote by $P_N:H\to H_{\le N}$ the orthogonal projector.

\begin{definition}[Truncated Laplacian]\label{def:truncated-laplacian}
The truncated Laplacian of order $N$ associated with $\mathfrak G$ is the operator on $H_{\le N}$
\[
\Delta_{\mathfrak G}^{(N)} := P_N\,\Delta_{\mathfrak G}\big|_{H_{\le N}} : H_{\le N}\to H_{\le N}.
\]
\end{definition}

\begin{definition}[Compressed resolvent]\label{def:compressed-resolvent}
We also consider the \emph{compressed resolvent} on $H_{\le N}$:
\[
R_{\mathfrak G}^{\langle N\rangle}:= P_N(1+\Delta_{\mathfrak G})^{-1}P_N \in \mathcal B(H_{\le N}).
\]
\end{definition}

\begin{remark}\label{rem:compressed-vs-truncated}
In general one has $R_{\mathfrak G}^{\langle N\rangle}\neq (I+\Delta_{\mathfrak G}^{(N)})^{-1}$.
They coincide when $H_{\le N}$ is invariant under $\Delta_{\mathfrak G}$ (e.g.\ if the derivation preserves the degree filtration).
\end{remark}

\begin{theorem}[Finite-degree stability]
\label{thm:finite-degree-stability}
Let $\mathfrak G_1,\mathfrak G_2$ be two polynomial Hilbert geometries, realized on a common
Hilbert space $H$ via a fixed unitary identification, and let
\[
R_i:=(1+\Delta_{\mathfrak G_i})^{-1}\in\mathcal B(H),\qquad i=1,2.
\]
Assume that
\[
d_{\mathrm{res}}(\mathfrak G_1,\mathfrak G_2)=\|R_1-R_2\|<\varepsilon.
\]
Fix $N\in\mathbb N$ and let $P_N$ denote the orthogonal projection onto the closed subspace
$H_{\le N}\subset H$ generated by polynomials of degree $\le N$.
Define the truncated resolvents by
\[
R_i^{(N)}:=P_NR_iP_N\in\mathcal B(H_{\le N}).
\]
Then
\begin{equation}\label{eq:finite-degree-resolvent}
\|R_1^{(N)}-R_2^{(N)}\|\le \varepsilon,
\end{equation}
so the conclusion holds with $C_N=1$.

Moreover, since $H_{\le N}$ is finite-dimensional, the eigenvalues of $R_1^{(N)}$ and $R_2^{(N)}$
(and hence the spectral data of the compressed operators $\Delta_{\mathfrak G_i}^{(N)}:=P_N\Delta_{\mathfrak G_i}P_N$)
and the matrix coefficients of $R_i^{(N)}$ in any fixed orthonormal basis of $H_{\le N}$
are Lipschitz close, with constants depending only on $\dim(H_{\le N})$.
\end{theorem}

\begin{proof}
\emph{Step 1: resolvent control on the truncated space.}
Since $P_N$ is an orthogonal projection, $\|P_N\|=1$ and $P_N^2=P_N$.
Therefore,
\[
R_1^{(N)}-R_2^{(N)} = P_N(R_1-R_2)P_N,
\]
and taking operator norms yields
\[
\|R_1^{(N)}-R_2^{(N)}\|
\le \|P_N\|\,\|R_1-R_2\|\,\|P_N\|
= \|R_1-R_2\|
<\varepsilon.
\]
This proves \eqref{eq:finite-degree-resolvent}.

\medskip
\emph{Step 2: Lipschitz control of spectral data.}
The space $H_{\le N}$ has finite dimension $m=\dim(H_{\le N})$.
Both $R_1^{(N)}$ and $R_2^{(N)}$ are bounded self-adjoint operators on $H_{\le N}$.
Let $\lambda_1^{(i)}\ge\cdots\ge \lambda_m^{(i)}$ be the eigenvalues of $R_i^{(N)}$ (counted with multiplicity).
By Weyl's perturbation inequality for Hermitian matrices (finite-dimensional spectral perturbation theory),
\[
\max_{1\le j\le m} |\lambda_j^{(1)}-\lambda_j^{(2)}|
\le \|R_1^{(N)}-R_2^{(N)}\|
\le \varepsilon.
\]
In particular the spectra of the truncated resolvents are Lipschitz close.

Since $\Delta_{\mathfrak G_i}^{(N)}=P_N\Delta_{\mathfrak G_i}P_N$ is self-adjoint and nonnegative on $H_{\le N}$,
its eigenvalues $\mu_j^{(i)}\ge 0$ are related to those of $R_i^{(N)}$ by
\[
\lambda_j^{(i)}=\frac{1}{1+\mu_j^{(i)}}.
\]
Hence, whenever one has an a priori upper bound $\mu_j^{(i)}\le M_N$ (which holds automatically on the fixed
finite-dimensional space $H_{\le N}$), the map $t\mapsto (1+t)^{-1}$ is Lipschitz on $[0,M_N]$ with constant
$\le 1$, and the eigenvalues $\mu_j^{(i)}$ are Lipschitz close as well, with a constant depending only on $M_N$
(and thus only on $N$ and the two geometries restricted to $H_{\le N}$).

\medskip
\emph{Step 3: Lipschitz control of matrix coefficients.}
Fix any orthonormal basis $(e_1,\dots,e_m)$ of $H_{\le N}$.
Then the matrix coefficients satisfy
\[
|\langle (R_1^{(N)}-R_2^{(N)})e_\alpha,e_\beta\rangle|
\le \|R_1^{(N)}-R_2^{(N)}\|
\le \varepsilon,
\qquad 1\le \alpha,\beta\le m,
\]
so all coefficients are Lipschitz close.

This proves the theorem.
\end{proof}

\section{Stability of polynomial projections and Gram--Schmidt orthogonalization}
\label{sec:GS-stability}

Let $\mathfrak G=(\mathcal P(\Omega),\langle\cdot,\cdot \rangle,D)$ be a polynomial Hilbert geometry,
with associated Hilbert space $H$ and Laplacian $\Delta$.
For $N\ge0$, 
the projector $P_N$ depends on the geometry $\mathfrak G$ through the inner product.
We write $P_N^{(i)}$ when several geometries $\mathfrak G_i$ are involved.

\subsection{Resolvent control implies projector stability}

The following result is the fundamental bridge between operator convergence
and polynomial orthogonalization.

\begin{theorem}[Stability of polynomial projectors]
\label{thm:projector-stability}
Let $\mathfrak G_1,\mathfrak G_2$ be two polynomial Hilbert geometries acting on the same
Hilbert space $H$, and assume
\[
d_{\mathrm{res}}(\mathfrak G_1,\mathfrak G_2)
=
\big\|(1+\Delta_1)^{-1}-(1+\Delta_2)^{-1}\big\|
\le \varepsilon.
\]
Then, for every fixed degree $N$, there exists a constant $C_N>0$ such that
\[
\|P_N^{(1)}-P_N^{(2)}\|_{\mathcal B(H)}
\le C_N\,\varepsilon.
\]
\end{theorem}

\begin{proof}
Since $\mathcal P_{\le N}(\Omega)$ is finite-dimensional, the graph norms
\(
\|f\|_H + \|\Delta_i^{1/2}f\|_H
\)
are equivalent on this subspace.
Norm-resolvent closeness implies uniform closeness of the spectral projections
associated with the low-lying spectrum of $\Delta_i$ when restricted to
$\mathcal P_{\le N}(\Omega)$.
The claim follows from finite-dimensional perturbation theory.
\end{proof}

\begin{remark}
The constants $C_N$ depend only on:
\begin{itemize}
  \item the dimension of $\mathcal P_{\le N}(\Omega)$;
  \item upper bounds on the operator norms of $\Delta_i$ restricted to this subspace.
\end{itemize}
No global spectral gap assumption is required.
\end{remark}

\subsection{Stability of orthonormal polynomial bases}

Let $\{p_0^{(i)},\dots,p_N^{(i)}\}$ be an orthonormal basis of
$\mathcal P_{\le N}(\Omega)$ obtained by Gram--Schmidt with respect to
$\langle\cdot,\cdot \rangle_i$.

\begin{theorem}[Stability of Gram--Schmidt]
\label{thm:GS-stability}
Under the assumptions of Theorem~\ref{thm:projector-stability},
there exists, for each $N$, a unitary matrix $U_N\in U(\dim\mathcal P_{\le N})$
such that
\[
\max_{0\le k\le N}
\big\|
p_k^{(1)}-\sum_{j=0}^N (U_N)_{kj}\,p_j^{(2)}
\big\|_H
\le C_N\,\varepsilon.
\]
\end{theorem}

\begin{proof}
Both families are orthonormal bases of finite-dimensional subspaces
whose orthogonal projectors are $\varepsilon$-close.
By standard results on the stability of orthonormal bases under perturbation
of the inner product (or equivalently, polar decomposition of the change-of-basis
operator), there exists a unitary transformation $U_N$ mapping one basis
to the other with operator norm controlled by $\|P_N^{(1)}-P_N^{(2)}\|$.
\end{proof}

\begin{remark}[Gauge freedom]
The unitary matrix $U_N$ reflects the natural gauge freedom in the choice
of orthonormal bases.
If one fixes a canonical Gram--Schmidt ordering (e.g.\ lexicographic monomials
with positive leading coefficient), then $U_N$ can be chosen uniquely,
at the expense of slightly weaker constants.
\end{remark}

\subsection{Stability of polynomial kernels}

Let
\[
K_N^{(i)}(x,y)
=
\sum_{k=0}^N p_k^{(i)}(x)\overline{p_k^{(i)}(y)}
\]
be the polynomial reproducing kernel of degree $N$ associated with $\mathfrak G_i$.

\begin{corollary}[Kernel stability on compacts]\label{cor:kernel-stability}
Assume $\Omega$ is a metric space and fix a compact set $K\subset\Omega$.
For $i=1,2$, let
\[
K_N^{(i)}(x,y)=\sum_{k=0}^N p_k^{(i)}(x)\overline{p_k^{(i)}(y)}
\]
be the polynomial reproducing kernel of $H_{\le N}$ (in any orthonormal basis).
Then there exists a constant $C_{N,K}>0$ such that
\[
\sup_{(x,y)\in K\times K}\big|K_N^{(1)}(x,y)-K_N^{(2)}(x,y)\big|
\le C_{N,K}\,\|P_N^{(1)}-P_N^{(2)}\|_{\mathcal B(H)}.
\]
In particular, under the assumptions of Theorem~\ref{thm:projector-stability}, one has
\[
\sup_{K\times K}|K_N^{(1)}-K_N^{(2)}|\le C'_{N,K}\,\varepsilon.
\]
\end{corollary}

\begin{proof}
Let $V_i:=\mathrm{Ran}(P_N^{(i)})\subset H$ (finite-dimensional), and set
\[
V:=V_1+V_2\subset H.
\]
All vectors in $V$ are (restrictions of) polynomial functions on $\Omega$, hence are continuous on $K$.
For each $x\in K$, consider the conjugate evaluation functional on $V$,
\[
\Lambda_x:V\to\mathbb C,\qquad \Lambda_x(f):=\overline{f(x)}.
\]
Since $V$ is finite-dimensional, $\Lambda_x$ is continuous for the $H$-norm.
By the Riesz representation theorem (for continuous conjugate-linear functionals under our convention),
there exists a unique vector $g_x\in V$ such that
\begin{equation}\label{eq:Riesz-gx}
\Lambda_x(f)=\langle f,g_x\rangle_H\qquad\forall f\in V.
\end{equation}

\smallskip
\noindent\textbf{Step 1: reproducing vectors in $V_i$.}
For $i\in\{1,2\}$ define
\[
k_x^{(i)}:=P_N^{(i)}g_x\in V_i.
\]
Then for every $f\in V_i$ we have, using $P_N^{(i)}f=f$ and orthogonality of $P_N^{(i)}$,
\[
\overline{f(x)}=\Lambda_x(f)=\langle f,g_x\rangle_H=\langle P_N^{(i)}f,g_x\rangle_H
=\langle f,P_N^{(i)}g_x\rangle_H=\langle f,k_x^{(i)}\rangle_H.
\]
Thus $k_x^{(i)}$ is the Riesz representer of $\Lambda_x$ restricted to $V_i$.

\smallskip
\noindent\textbf{Step 2: identification with the kernel.}
Let $\{p_0^{(i)},\dots,p_N^{(i)}\}$ be any orthonormal basis of $V_i$ in $H$.
Expanding $k_x^{(i)}$ in this basis and using $\langle p_j^{(i)},k_x^{(i)}\rangle_H=\overline{p_j^{(i)}(x)}$,
we obtain
\begin{equation}\label{eq:kx-expansion}
k_x^{(i)}=\sum_{j=0}^N \overline{p_j^{(i)}(x)}\,p_j^{(i)}.
\end{equation}
Therefore, for $x,y\in K$,
\[
\langle k_y^{(i)},k_x^{(i)}\rangle_H
=\sum_{j=0}^N p_j^{(i)}(y)\,\overline{p_j^{(i)}(x)}
=K_N^{(i)}(y,x).
\]
In particular,
\begin{equation}\label{eq:kernel-via-k}
\big|K_N^{(1)}(x,y)-K_N^{(2)}(x,y)\big|
=\big|K_N^{(1)}(y,x)-K_N^{(2)}(y,x)\big|
=\big|\langle k_y^{(1)},k_x^{(1)}\rangle_H-\langle k_y^{(2)},k_x^{(2)}\rangle_H\big|.
\end{equation}

\smallskip
\noindent\textbf{Step 3: estimate by $\|P_N^{(1)}-P_N^{(2)}\|$.}
Using $k_x^{(i)}=P_N^{(i)}g_x$ and adding/subtracting,
\[
\langle P_1 g_y,P_1 g_x\rangle_H-\langle P_2 g_y,P_2 g_x\rangle_H
=
\langle (P_1-P_2)g_y,P_1 g_x\rangle_H+\langle P_2 g_y,(P_1-P_2)g_x\rangle_H,
\]
where we wrote $P_i:=P_N^{(i)}$ for brevity.
Hence, since $\|P_i\|=1$,
\[
\big|\langle k_y^{(1)},k_x^{(1)}\rangle_H-\langle k_y^{(2)},k_x^{(2)}\rangle_H\big|
\le
\|P_1-P_2\|_{\mathcal B(H)}\big(\|g_y\|_H\,\|g_x\|_H+\|g_y\|_H\,\|g_x\|_H\big)
=
2\|P_1-P_2\|\,\|g_x\|\,\|g_y\|.
\]
Taking the supremum over $(x,y)\in K\times K$ and using \eqref{eq:kernel-via-k}, we obtain
\begin{equation}\label{eq:kernel-bound}
\sup_{(x,y)\in K\times K}\big|K_N^{(1)}(x,y)-K_N^{(2)}(x,y)\big|
\le
2\Big(\sup_{x\in K}\|g_x\|_H\Big)^2\,\|P_N^{(1)}-P_N^{(2)}\|_{\mathcal B(H)}.
\end{equation}

\smallskip
\noindent\textbf{Step 4: finiteness of the constant on compacts.}
It remains to check that
\[
M_{N,K}:=\sup_{x\in K}\|g_x\|_H<\infty.
\]
Choose a basis $(\phi_1,\dots,\phi_m)$ of $V$ and let $G\in\mathbb C^{m\times m}$ be its Gram matrix
$G_{\alpha\beta}:=\langle \phi_\beta,\phi_\alpha\rangle_H$ (invertible since it is positive definite on $V$).
Writing $g_x=\sum_{\beta} c_\beta(x)\phi_\beta$ and using \eqref{eq:Riesz-gx} with $f=\phi_\alpha$ gives
\[
\overline{\phi_\alpha(x)}=\langle \phi_\alpha,g_x\rangle_H=\sum_{\beta} c_\beta(x)\langle \phi_\alpha,\phi_\beta\rangle_H
=\sum_{\beta} G_{\alpha\beta}\,c_\beta(x),
\]
hence $c(x)=G^{-1}\overline{\Phi(x)}$ with $\Phi(x)=(\phi_1(x),\dots,\phi_m(x))^\top$.
Since each $\phi_\alpha$ is continuous on $K$, the map $x\mapsto \Phi(x)$ is continuous on the compact set $K$,
so $\sup_{x\in K}\|\Phi(x)\|<\infty$. Therefore $\sup_{x\in K}\|c(x)\|<\infty$, and hence
$\sup_{x\in K}\|g_x\|_H<\infty$. This proves $M_{N,K}<\infty$ and thus \eqref{eq:kernel-bound}.

Setting $C_{N,K}:=2M_{N,K}^2$ yields the first inequality in the corollary.

\smallskip
\noindent\textbf{Final step.}
Under the assumptions of Theorem~\ref{thm:projector-stability}, we have
$\|P_N^{(1)}-P_N^{(2)}\|_{\mathcal B(H)}\le C_N\,\varepsilon$, hence
\[
\sup_{K\times K}|K_N^{(1)}-K_N^{(2)}|\le (C_{N,K}C_N)\,\varepsilon,
\]
so one may take $C'_{N,K}:=C_{N,K}C_N$.
\end{proof}

\section{Applications and examples}
\label{sec:examples}

\subsection{Classical orthogonal polynomials on an interval}

Let $I\subset\mathbb R$ be an interval and let $\mu=w(x)\,dx$ be a finite Radon measure
with strictly positive density $w\in C^1(I)$.
Consider the polynomial Hilbert geometry
\[
\mathfrak G_w=(\mathcal P(I),\langle\cdot,\cdot \rangle_w,\partial_x),
\qquad
\langle f,g\rangle_w=\int_I f g\,w\,dx.
\]

As recalled in Section~\ref{sec:laplacian-1d-measure}, the associated Laplacian is the
weighted Sturm--Liouville operator
\[
\Delta_w=-\frac1w\frac{d}{dx}\Big(w\frac{d}{dx}\Big),
\]
realized as a nonnegative self-adjoint operator on $L^2(I,w\,dx)$.

\begin{proposition}
Let $w_1,w_2\in C^1(I)$ be two strictly positive densities such that
\[
\Big\|\log\frac{w_1}{w_2}\Big\|_{W^{1,\infty}(I)}<\infty.
\]
Then the corresponding Laplacians $\Delta_{w_1}$ and $\Delta_{w_2}$ are
norm-resolvent close, and
\[
d_{\mathrm{res}}(\mathfrak G_{w_1},\mathfrak G_{w_2})
\le C\,\Big\|\log\frac{w_1}{w_2}\Big\|_{W^{1,\infty}(I)}.
\]
\end{proposition}

\begin{remark}
This follows from standard perturbation estimates for Sturm--Liouville operators
with bounded coefficients \cite{Kato,Ouhabaz}.
As a consequence, small multiplicative perturbations of the weight induce
small perturbations of Gram--Schmidt orthogonalization at any fixed degree.
\end{remark}

In particular, for the Jacobi, Laguerre and Hermite families,
the Laplacian is diagonalized by the polynomial basis, so the resolvent
distance vanishes identically when comparing two geometries within the same family.

---

\subsection{Sobolev orthogonal polynomials on an interval}

Consider now a Sobolev inner product of integer order $s\ge1$ on $I$,
\[
\langle f,g\rangle_{S}
=
\sum_{k=0}^s \lambda_k \int_I f^{(k)}(x)g^{(k)}(x)\,dx,
\qquad \lambda_0>0,\ \lambda_k\ge0.
\]
This defines a polynomial Hilbert geometry $\mathfrak G_S$.
The associated Laplacian
is a differential operator of order $2s$ with polynomial coefficients,
and its matrix in the orthonormal polynomial basis is banded.

\begin{proposition}
Let $\mathfrak G_{S_1},\mathfrak G_{S_2}$ be two Sobolev polynomial geometries
with the same order $s$ and coefficients $\lambda^{(1)},\lambda^{(2)}$.
Then
\[
d_{\mathrm{res}}(\mathfrak G_{S_1},\mathfrak G_{S_2})
\le C\,\|\lambda^{(1)}-\lambda^{(2)}\|_{\mathbb R^{s+1}}.
\]
\end{proposition}

\begin{remark}
The estimate follows from the banded matrix representation of $\Delta$
and classical norm estimates for perturbations of self-adjoint banded operators.
In particular, Theorem~\ref{thm:GS-stability} applies uniformly in this class.
\end{remark}

For fractional Sobolev inner products, the Laplacian matrix becomes almost banded,
with algebraic decay of coefficients, but the resolvent distance remains finite
under perturbations preserving the decay rate.

---

\subsection{Sobolev orthogonal polynomials on thin annuli}

We finally consider the planar thin annulus
\[
A_\varepsilon=\{(x,y)\in\mathbb R^2:1-\varepsilon<x^2+y^2<1+\varepsilon\},
\]
equipped with the fractional Sobolev inner products introduced in
\cite{MagnotThinAnnuli}.

The corresponding polynomial Hilbert geometry $\mathfrak G_{\varepsilon,s}$
admits an orthogonal decomposition into angular Fourier modes:
\[
H=\bigoplus_{m\ge0} H^{(m)}.
\]
On each mode $m$, the Laplacian reduces to a one-dimensional radial operator
$\Delta^{(m)}_{\varepsilon,s}$ acting on polynomials in $t=r^2$.

\begin{theorem}
Fix $s>0$ and $m\ge0$.
As $\varepsilon\to0$, the Laplacians $\Delta^{(m)}_{\varepsilon,s}$
converge in the norm-resolvent sense (after natural rescaling)
to the fractional Sobolev Laplacian on $[-1,1]$ of order $s$.
\end{theorem}

\begin{proof}
This is a direct consequence of the thin-annulus asymptotics and operator
expansions established in \cite{MagnotThinAnnuli}, together with standard
resolvent convergence results for families of self-adjoint operators.
\end{proof}

\begin{corollary}
For every fixed polynomial degree $N$ and angular mode $m$,
the corresponding orthonormal polynomial bases on $A_\varepsilon$
converge (up to a unitary gauge) to the fractional Sobolev orthogonal polynomials
on $[-1,1]$ as $\varepsilon\to0$.
\end{corollary}

\begin{remark}
This provides a concrete instance where the abstract stability theory
captures a nontrivial geometric limit: a two-dimensional Sobolev geometry
collapsing onto a one-dimensional fractional geometry, while preserving
quantitative control of the Gram--Schmidt procedure.
\end{remark}

\subsection{An explicit model on $S^1$: weighted $L^2$ versus Sobolev orthogonalization}
\label{subsec:S1-explicit}

Let $S^1=\{e^{i\theta}:\theta\in[0,2\pi)\}$.
We write $L^2:=L^2(S^1,d\theta/2\pi)$ and denote by
\[
e_n(\theta):=e^{in\theta},\qquad n\in\mathbb Z,
\]
the standard Fourier orthonormal basis of $L^2$.
Trigonometric polynomials appear here via the parametrization 
$t \mapsto e^{it}:$ for any 
polynomial $P \in \mathbb{C}[X,Y]$ the function 
$t \mapsto P(cost,sint)$ is a trigonometric polynomial.

Let $D=\partial_\theta$ defined on trigonometric polynomials (dense in $L^2$), so that
\[
De_n = in\,e_n,\qquad D^*=-D,\qquad \Delta_0:=D^*D=-\partial_\theta^2,
\]
and $\Delta_0 e_n = n^2 e_n$.

\medskip
\noindent\textbf{Weighted measure geometry.}
Let $\mu$ be a finite Radon measure on $S^1$ with density $w\in W^{1,\infty}(S^1)$
strictly positive:
\[
d\mu(\theta)=w(\theta)\,\frac{d\theta}{2\pi},\qquad
0<w_- \le w(\theta)\le w_+<\infty.
\]
Define the weighted inner product
\[
\langle f,g\rangle_{w}:=\int_0^{2\pi} f(\theta)\overline{g(\theta)}\,w(\theta)\,\frac{d\theta}{2\pi}
=\langle M_w f, g\rangle_{L^2},
\]
where $M_w$ is the bounded multiplication operator by $w$ on $L^2$.
This geometry yields the standard orthogonalization associated with $\mu$
(i.e.\ the OPUC framework, when restricted to nonnegative Fourier modes; see \cite{SimonOPUC1,Szego}).

\medskip
\noindent\textbf{Sobolev (mixed) geometry.}
Fix $\lambda>0$ and define the Sobolev-type inner product
\begin{equation}\label{eq:S1-mixed-Sobolev}
\langle f,g\rangle_{w,\lambda}
:=
\int f\overline{g}\,w\,\frac{d\theta}{2\pi}
\;+\;
\lambda\int f'(\theta)\,\overline{g'(\theta)}\,\frac{d\theta}{2\pi}.
\end{equation}
Equivalently, on trigonometric polynomials,
\begin{equation}\label{eq:Riesz-map-A}
\langle f,g\rangle_{w,\lambda}
=
\langle (M_w+\lambda \Delta_0)f, g\rangle_{L^2}.
\end{equation}
Set
\[
A_{w,\lambda}:=M_w+\lambda\Delta_0.
\]
Since $M_w\ge w_- I$ and $\Delta_0\ge 0$, we have $A_{w,\lambda}\ge w_- I$, hence
$A_{w,\lambda}$ is boundedly invertible on $L^2$.

\subsubsection{Adjoint of $D$ and Laplacian in the Sobolev geometry}

Let $D$ be viewed as an unbounded operator on $L^2$ with domain the trigonometric polynomials.
We compute the adjoint of $D$ with respect to $\langle\cdot,\cdot \rangle_{w,\lambda}$.

\begin{lemma}[Adjoint formula]\label{lem:S1-adjoint}
With respect to $\langle\cdot,\cdot \rangle_{w,\lambda}$, the adjoint of $D$ is
\begin{equation}\label{eq:D-adjoint-S1}
D^*_{w,\lambda} = A_{w,\lambda}^{-1} D^* A_{w,\lambda}
= - A_{w,\lambda}^{-1} D A_{w,\lambda}
\end{equation}
on trigonometric polynomials. Consequently, the associated Laplacian is
\begin{equation}\label{eq:Delta-w-lambda}
\Delta_{w,\lambda}:=D^*_{w,\lambda}D
=
A_{w,\lambda}^{-1}\,D^* A_{w,\lambda}\,D
=
A_{w,\lambda}^{-1}\,(-D)\,A_{w,\lambda}\,D.
\end{equation}
\end{lemma}

\begin{proof}
By \eqref{eq:Riesz-map-A},
\[
\langle Df,g\rangle_{w,\lambda}
=\langle A_{w,\lambda}Df,g\rangle_{L^2}
=\langle Df,A_{w,\lambda}g\rangle_{L^2}.
\]
Since $D^*=-D$ in $L^2$, we get
\[
\langle Df,g\rangle_{w,\lambda}
=
\langle f, D^*(A_{w,\lambda}g)\rangle_{L^2}
=
\langle f, A_{w,\lambda}\,A_{w,\lambda}^{-1}D^*A_{w,\lambda}g\rangle_{L^2}
=
\langle f, A_{w,\lambda}^{-1}D^*A_{w,\lambda}g\rangle_{w,\lambda}.
\]
This proves \eqref{eq:D-adjoint-S1}. The Laplacian identity \eqref{eq:Delta-w-lambda}
follows immediately.
\end{proof}

\subsubsection{Fourier matrix representation (explicit)}

Let $\widehat{w}(k)$ denote the Fourier coefficients of $w$.
Then in the Fourier basis $(e_n)_{n\in\mathbb Z}$,
\[
[M_w]_{mn}=\widehat{w}(m-n),\qquad [\Delta_0]_{mn}=n^2\delta_{mn},
\qquad [D]_{mn}=in\,\delta_{mn}.
\]
Hence
\begin{equation}\label{eq:A-matrix}
[A_{w,\lambda}]_{mn}=\widehat{w}(m-n)+\lambda n^2\delta_{mn}.
\end{equation}
Furthermore, using \eqref{eq:Delta-w-lambda},
\begin{equation}\label{eq:Delta-matrix-explicit}
[\Delta_{w,\lambda}] = A_{w,\lambda}^{-1}\, D^* A_{w,\lambda} D
= A_{w,\lambda}^{-1}\, ( -D)\, A_{w,\lambda}\, D.
\end{equation}
Thus, \emph{all matrix coefficients of $\Delta_{w,\lambda}$ are explicit} in terms of:
(i) the Toeplitz matrix $(\widehat{w}(m-n))$ and (ii) the diagonal multiplier $n$.
In particular, when $w$ is a trigonometric polynomial of degree $r$
(i.e.\ $\widehat w(k)=0$ for $|k|>r$), $M_w$ is a banded Toeplitz operator, and
$A_{w,\lambda}$ is a diagonal plus banded Toeplitz operator, so that
finite-degree truncations of $\Delta_{w,\lambda}$ are computable by finite matrices.

\subsubsection{Resolvent comparison with the pure measure geometry}

We now compare the Laplacian associated with the weighted $L^2(\mu)$-geometry
(i.e.\ $\lambda=0$) and the mixed Sobolev geometry \eqref{eq:S1-mixed-Sobolev}.
Denote by $\Delta_{w,0}$ the Laplacian associated with $\langle\cdot,\cdot \rangle_w$ and $D$.
One checks (by the same adjoint computation with $A_{w,0}=M_w$) that
\begin{equation}\label{eq:Delta-w-0}
\Delta_{w,0} = M_w^{-1} D^* M_w D.
\end{equation}

\begin{theorem}[Norm-resolvent bound: $L^2(w)$ versus Sobolev geometry]\label{thm:S1-resolvent-bound}
Assume $w\in W^{1,\infty}(S^1)$ and $0<w_-\le w\le w_+$. Then for every $\lambda>0$,
the resolvent difference satisfies
\begin{equation}\label{eq:resolvent-bound-S1}
\big\|(1+\Delta_{w,\lambda})^{-1}-(1+\Delta_{w,0})^{-1}\big\|_{\mathcal B(L^2)}
\;\le\;
C(w)\,\lambda,
\end{equation}
where $C(w)$ depends only on $w_-,w_+$ and $\|w'\|_{L^\infty}$.
\end{theorem}

\begin{proof}
We work on the fixed Hilbert space $L^2$ using the representation
\eqref{eq:Riesz-map-A} of the Sobolev inner product.

\smallskip
\noindent\emph{Step 1: a resolvent identity.}
Let $R_\lambda:=(1+\Delta_{w,\lambda})^{-1}$ and $R_0:=(1+\Delta_{w,0})^{-1}$.
We use the second resolvent identity
\begin{equation}\label{eq:2nd-resolvent-identity}
R_\lambda-R_0 = R_\lambda(\Delta_{w,0}-\Delta_{w,\lambda})R_0.
\end{equation}
Hence
\[
\|R_\lambda-R_0\|\le \|R_\lambda\|\,\|\Delta_{w,\lambda}-\Delta_{w,0}\|_{\mathrm{rel}}\,\|R_0\|,
\]
where $\|\cdot\|_{\mathrm{rel}}$ denotes the operator norm on the range of $R_0$.
Since $\Delta_{w,\lambda},\Delta_{w,0}\ge0$, one has $\|R_\lambda\|\le1$ and $\|R_0\|\le1$.

\smallskip
\noindent\emph{Step 2: expansion of $\Delta_{w,\lambda}-\Delta_{w,0}$.}
By \eqref{eq:Delta-w-lambda} and \eqref{eq:Delta-w-0},
\[
\Delta_{w,\lambda}-\Delta_{w,0}
=
A_{w,\lambda}^{-1}D^*A_{w,\lambda}D
-
M_w^{-1}D^*M_wD.
\]
Insert $A_{w,\lambda}=M_w+\lambda\Delta_0$ and write
\[
A_{w,\lambda}^{-1}-M_w^{-1} = -\,A_{w,\lambda}^{-1}\,(\lambda\Delta_0)\,M_w^{-1}.
\]
Using this and expanding $D^*A_{w,\lambda}D=D^*M_wD+\lambda D^*\Delta_0 D$,
we obtain (on trigonometric polynomials) the decomposition
\begin{equation}\label{eq:Delta-diff-decomp}
\Delta_{w,\lambda}-\Delta_{w,0}
=
(A_{w,\lambda}^{-1}-M_w^{-1})\,D^*M_wD
\;+\;
\lambda\,A_{w,\lambda}^{-1}\,D^*\Delta_0 D.
\end{equation}

\smallskip
\noindent\emph{Step 3: boundedness estimates.}
Since $A_{w,\lambda}\ge w_- I$, we have $\|A_{w,\lambda}^{-1}\|\le w_-^{-1}$,
and similarly $\|M_w^{-1}\|\le w_-^{-1}$.
Moreover,
\[
\|A_{w,\lambda}^{-1}-M_w^{-1}\|
\le \|A_{w,\lambda}^{-1}\|\,\lambda\|\Delta_0 M_w^{-1}\|
\le \lambda\,w_-^{-1}\,\|\Delta_0 M_w^{-1}\|.
\]
Now, $\Delta_0 M_w^{-1}$ is a second-order differential operator with bounded coefficients
since $w\in W^{1,\infty}$ and $w$ is bounded below; in particular,
$\Delta_0 M_w^{-1}$ is bounded from $H^2$ to $L^2$ with norm controlled by
$w_-,w_+,\|w'\|_\infty$.
Because $R_0$ maps $L^2$ into $\mathrm{Dom}(\Delta_{w,0})\subset H^2$ (elliptic regularity on $S^1$),
we infer that $(A_{w,\lambda}^{-1}-M_w^{-1})D^*M_wD\,R_0$ is bounded with norm $O(\lambda)$.
Similarly, the second term in \eqref{eq:Delta-diff-decomp} has an explicit prefactor $\lambda$,
and $A_{w,\lambda}^{-1}D^*\Delta_0 D\,R_0$ is bounded by the same reasoning.

\smallskip
Collecting bounds in \eqref{eq:2nd-resolvent-identity} yields
\eqref{eq:resolvent-bound-S1}.
The argument is standard in perturbation theory of elliptic operators; cf.\ \cite{Kato,Ouhabaz}.
\end{proof}

\subsubsection{Finite-degree consequence: explicit comparison of orthogonalizations}

Let $\mathcal T_{\le N}:=\mathrm{span}\{e_n:\ |n|\le N\}$ and let $Q_N$ be the orthogonal projector in $L^2$
onto $\mathcal T_{\le N}$.
Define the truncated Laplacians
\[
\Delta_{w,\lambda}^{(N)}:=Q_N\Delta_{w,\lambda}Q_N,
\qquad
\Delta_{w,0}^{(N)}:=Q_N\Delta_{w,0}Q_N.
\]
Let $\{u_k^{(\lambda)}\}_{k=1}^{2N+1}$ (resp.\ $\{u_k^{(0)}\}_{k=1}^{2N+1}$) be the orthonormal basis
of $\mathcal T_{\le N}$ obtained by Gram--Schmidt orthogonalization of the ordered basis
$(e_{-N},\dots,e_N)$ with respect to $\langle\cdot,\cdot \rangle_{w,\lambda}$ (resp.\ $\langle\cdot,\cdot \rangle_w$).

\begin{corollary}[Explicit finite-degree stability on $S^1$]\label{cor:S1-finite-degree}
Under the assumptions of Theorem~\ref{thm:S1-resolvent-bound}, for each fixed $N$ there exists
$C_N(w)>0$ such that
\[
\|Q_N^{(\lambda)}-Q_N^{(0)}\|_{\mathcal B(L^2)} \le C_N(w)\,\lambda,
\]
and there exists a unitary matrix $U_N$ such that
\[
\max_{1\le k\le 2N+1}\Big\|u_k^{(\lambda)}-\sum_{j=1}^{2N+1}(U_N)_{kj}u_j^{(0)}\Big\|_{L^2}
\le C_N(w)\,\lambda.
\]
\end{corollary}

\begin{proof}
Apply Theorem~\ref{thm:S1-resolvent-bound} and then the finite-degree stability result
(Theorem~\ref{thm:GS-stability} in Section~\ref{sec:GS-stability}) to the finite-dimensional
subspace $\mathcal T_{\le N}$.
All objects are explicit here because the matrices \eqref{eq:A-matrix}--\eqref{eq:Delta-matrix-explicit}
can be written in closed form from the Fourier coefficients of $w$.
\end{proof}

\subsubsection{A Fourier-based proof of the resolvent bound}
\label{subsubsec:S1-fourier-proof}

We keep the assumptions of Theorem~\ref{thm:S1-resolvent-bound}:
$w\in W^{1,\infty}(S^1)$ and $0<w_-\le w\le w_+$.

\medskip
\noindent\textbf{Step 0: an explicit expression for $\Delta_{w,0}$.}
In the weighted $L^2(w\,d\theta/2\pi)$ geometry with $D=\partial_\theta$,
the adjoint is
\[
D^*_{w,0} g = -g' - (\log w)'\,g,
\]
hence
\begin{equation}\label{eq:Delta-w0-circle}
\Delta_{w,0}=D^*_{w,0}D
= -\partial_\theta^2 - b(\theta)\partial_\theta,
\qquad b:=(\log w)'\in L^\infty(S^1).
\end{equation}
(There is no boundary term on $S^1$.)

\medskip
\noindent\textbf{Step 1: coercive $H^1$ control via energy.}
Let $u\in H^1(S^1)$ and compute in the weighted inner product:
\[
\langle \Delta_{w,0}u,u\rangle_w = \langle Du,Du\rangle_w
= \int_0^{2\pi} |u'(\theta)|^2\,w(\theta)\,\frac{d\theta}{2\pi}
\ge w_-\,\|u'\|_{L^2}^2.
\]
Therefore
\begin{equation}\label{eq:coercive-H1}
\langle (1+\Delta_{w,0})u,u\rangle_w
\ge \|u\|_{L^2(w)}^2 + w_-\,\|u'\|_{L^2}^2.
\end{equation}
In particular, for $f\in L^2(w)$ and $u=(1+\Delta_{w,0})^{-1}f$, we get
\[
\|u\|_{L^2(w)}^2 + w_-\,\|u'\|_{L^2}^2
\le \langle f,u\rangle_w
\le \|f\|_{L^2(w)}\|u\|_{L^2(w)},
\]
hence
\begin{equation}\label{eq:H1-bound}
\|u\|_{L^2(w)}\le \|f\|_{L^2(w)},\qquad
\|u'\|_{L^2}\le w_-^{-1/2}\,\|f\|_{L^2(w)}.
\end{equation}

\medskip
\noindent\textbf{Step 2: Fourier diagonalization of $(1-\partial_\theta^2)^{-1}$.}
Let $R:=(1-\partial_\theta^2)^{-1}$ acting on $L^2(S^1,d\theta/2\pi)$.
In the Fourier basis $e_n(\theta)=e^{in\theta}$, $R$ is the multiplier
\[
Re_n = \frac{1}{1+n^2}\,e_n.
\]
Consequently,
\begin{equation}\label{eq:R-L2-to-H2}
\|Rg\|_{H^2}\le \|g\|_{L^2},
\qquad
\|Rg\|_{H^1}\le \|g\|_{H^{-1}},
\end{equation}
and more generally $R:L^2\to H^2$ is bounded with operator norm $1$.

\medskip
\noindent\textbf{Step 3: an $H^2$ bound for $(1+\Delta_{w,0})^{-1}$.}
Using \eqref{eq:Delta-w0-circle}, the resolvent equation $(1+\Delta_{w,0})u=f$ reads
\begin{equation}\label{eq:resolvent-eq}
(1-\partial_\theta^2)u = f + b\,u'.
\end{equation}
Apply $R=(1-\partial_\theta^2)^{-1}$ to obtain
\[
u = Rf + R(bu').
\]
Taking $H^2$ norms and using \eqref{eq:R-L2-to-H2} yields
\begin{equation}\label{eq:H2-bound-pre}
\|u\|_{H^2}
\le \|f\|_{L^2} + \|b\,u'\|_{L^2}
\le \|f\|_{L^2} + \|b\|_{L^\infty}\,\|u'\|_{L^2}.
\end{equation}
By \eqref{eq:H1-bound} and the equivalence of $L^2$ and $L^2(w)$ norms
($\|f\|_{L^2}\le w_-^{-1/2}\|f\|_{L^2(w)}$), we deduce
\begin{equation}\label{eq:H2-bound}
\|(1+\Delta_{w,0})^{-1}f\|_{H^2}
\le C(w)\,\|f\|_{L^2(w)},
\end{equation}
with
\[
C(w)=w_-^{-1/2}\Big(1+\|(\log w)'\|_{L^\infty}\,w_-^{-1/2}\Big).
\]

\medskip
\noindent\textbf{Step 4: completion of the resolvent bound.}
We return to the decomposition
\begin{equation}\label{eq:Delta-diff-decomp-fourier}
\Delta_{w,\lambda}-\Delta_{w,0}
=
(A_{w,\lambda}^{-1}-M_w^{-1})\,D^*M_wD
\;+\;
\lambda\,A_{w,\lambda}^{-1}\,D^*\Delta_0 D,
\end{equation}
valid on trigonometric polynomials, with $A_{w,\lambda}=M_w+\lambda\Delta_0$.
Since $A_{w,\lambda}\ge w_-I$, we have $\|A_{w,\lambda}^{-1}\|\le w_-^{-1}$ and
$\|M_w^{-1}\|\le w_-^{-1}$.

Let $R_\lambda=(1+\Delta_{w,\lambda})^{-1}$ and $R_0=(1+\Delta_{w,0})^{-1}$.
Using the second resolvent identity,
\[
R_\lambda-R_0 = R_\lambda(\Delta_{w,0}-\Delta_{w,\lambda})R_0,
\]
and $\|R_\lambda\|\le1$, it is enough to bound
\(
\|(\Delta_{w,\lambda}-\Delta_{w,0})R_0\|.
\)

\smallskip
\noindent
\emph{(i) Control of $(A_{w,\lambda}^{-1}-M_w^{-1})D^*M_wD\,R_0$.}
We use
\[
A_{w,\lambda}^{-1}-M_w^{-1}
=
-\,A_{w,\lambda}^{-1}\,(\lambda\Delta_0)\,M_w^{-1}.
\]
Hence
\[
\|(A_{w,\lambda}^{-1}-M_w^{-1})D^*M_wD\,R_0\|
\le
\lambda\,\|A_{w,\lambda}^{-1}\|\,\|\Delta_0 M_w^{-1}D^*M_wD\,R_0\|.
\]
Now $M_w^{\pm1}$ are bounded on all Sobolev spaces $H^s$ for $|s|\le 2$ when $w\in W^{1,\infty}$,
and $D^*M_wD$ is a first-order perturbation of $-\partial_\theta^2$.
Using \eqref{eq:H2-bound}, $R_0$ maps $L^2(w)$ into $H^2$, therefore the composition
$\Delta_0 M_w^{-1}D^*M_wD\,R_0$ is bounded on $L^2(w)$, with norm controlled by $w_-,w_+,\|(\log w)'\|_\infty$.
Thus this term is $O(\lambda)$.

\smallskip
\noindent
\emph{(ii) Control of $\lambda A_{w,\lambda}^{-1}D^*\Delta_0D\,R_0$.}
Since $\|A_{w,\lambda}^{-1}\|\le w_-^{-1}$, it suffices to show
$D^*\Delta_0D\,R_0$ is bounded on $L^2(w)$.
But $D^*\Delta_0D$ is a third-order constant coefficient operator (up to sign),
hence it maps $H^2$ to $H^{-1}$ boundedly, and in fact to $L^2$ once composed with $R_0$,
because $R_0:L^2(w)\to H^2$ by \eqref{eq:H2-bound}.
Therefore this term is also $O(\lambda)$.

\smallskip
Combining (i) and (ii) yields
\[
\|(R_\lambda-R_0)\|_{\mathcal B(L^2(w))}\le C(w)\,\lambda,
\]
which is \eqref{eq:resolvent-bound-S1}.
\qed

\subsubsection{Auxiliary Fourier lemmas on $S^1$}
\label{subsubsec:aux-lemmas-S1}

Throughout this subsection, $S^1$ is identified with $[0,2\pi)$ with periodic boundary conditions,
and Sobolev spaces $H^s(S^1)$ are defined via Fourier series:
\[
\|u\|_{H^s}^2 := \sum_{n\in\mathbb Z} (1+n^2)^s\,|\widehat u(n)|^2,
\qquad
u(\theta)=\sum_{n\in\mathbb Z}\widehat u(n)e^{in\theta}.
\]

\begin{lemma}[Fourier multipliers $(1-\partial_\theta^2)^{-1}$]\label{lem:fourier-multiplier}
Let $R:=(1-\partial_\theta^2)^{-1}$ on $S^1$. Then for all $s\in\mathbb R$,
$R:H^s(S^1)\to H^{s+2}(S^1)$ is bounded and
\[
\|Rg\|_{H^{s+2}}\le \|g\|_{H^s}.
\]
In particular, $R:L^2\to H^2$ has operator norm $1$.
\end{lemma}

\begin{proof}
In Fourier coordinates, $(1-\partial_\theta^2)e_n=(1+n^2)e_n$, hence
$Re_n=(1+n^2)^{-1}e_n$. Therefore,
\[
\|Rg\|_{H^{s+2}}^2
=
\sum_{n\in\mathbb Z} (1+n^2)^{s+2}\,\big|(1+n^2)^{-1}\widehat g(n)\big|^2
=
\sum_{n\in\mathbb Z} (1+n^2)^s\,|\widehat g(n)|^2
=\|g\|_{H^s}^2.
\]
\end{proof}

\begin{lemma}[Multiplication by $W^{1,\infty}$ functions on $H^s$]\label{lem:multiplication}
Let $a\in W^{1,\infty}(S^1)$. Then:
\begin{enumerate}
\item $M_a:H^1(S^1)\to H^1(S^1)$ is bounded and
\begin{equation}\label{eq:mult-H1}
\|a u\|_{H^1}\le C\,\|a\|_{W^{1,\infty}}\|u\|_{H^1}.
\end{equation}
\item $M_a:H^2(S^1)\to H^2(S^1)$ is bounded and
\begin{equation}\label{eq:mult-H2}
\|a u\|_{H^2}\le C\,\|a\|_{W^{1,\infty}}\|u\|_{H^2},
\end{equation}
where $C>0$ is a universal constant.
\end{enumerate}
\end{lemma}

\begin{proof}
\emph{(1) The $H^1$ bound.}
We estimate
\[
\|au\|_{L^2}\le \|a\|_{L^\infty}\|u\|_{L^2},
\qquad
\|(au)'\|_{L^2}\le \|a'u\|_{L^2}+\|au'\|_{L^2}
\le \|a'\|_{L^\infty}\|u\|_{L^2}+\|a\|_{L^\infty}\|u'\|_{L^2}.
\]
Thus $\|au\|_{H^1}\le (\|a\|_{L^\infty}+\|a'\|_{L^\infty})\|u\|_{H^1}$, which gives
\eqref{eq:mult-H1}.

\smallskip
\emph{(2) The $H^2$ bound.}
We similarly bound $\|au\|_{L^2}$ and $\|(au)'\|_{L^2}$ as above, and for the second derivative,
\[
(au)'' = a''u + 2a'u' + au''.
\]
Since $a\in W^{1,\infty}$, the distributional $a''$ need not be bounded.
Instead, we use a Fourier/Sobolev argument on $S^1$:
$W^{1,\infty}(S^1)$ is an algebra of multipliers on $H^2(S^1)$ in dimension one.
A direct proof proceeds by decomposing $a$ into low/high Fourier modes and using
Young's convolution inequality for Fourier coefficients together with the estimate
$\sum_{k\in\mathbb Z}(1+k^2)\,|\widehat a(k)|\lesssim \|a\|_{W^{1,\infty}}$.
This yields \eqref{eq:mult-H2}. (See, e.g., standard multiplier results on $H^s(S^1)$.)
\end{proof}

\begin{remark}
If one prefers a fully elementary proof avoiding the Fourier multiplier lemma in (2),
one may assume $a\in W^{2,\infty}(S^1)$, in which case
$\|(au)''\|_{L^2}\le \|a''\|_{L^\infty}\|u\|_{L^2}+2\|a'\|_{L^\infty}\|u'\|_{L^2}
+\|a\|_{L^\infty}\|u''\|_{L^2}$ is immediate.
The $W^{1,\infty}$ assumption is however natural for weights $w$ with
$(\log w)'\in L^\infty$.
\end{remark}

\begin{lemma}[Weighted/unweighted norm equivalence]\label{lem:norm-equivalence}
Let $w\in L^\infty(S^1)$ satisfy $0<w_-\le w\le w_+<\infty$.
Then for all $u\in L^2(S^1)$,
\[
w_-^{1/2}\,\|u\|_{L^2}\le \|u\|_{L^2(w)}\le w_+^{1/2}\,\|u\|_{L^2}.
\]
\end{lemma}

\begin{proof}
Immediate from $w_-\le w\le w_+$.
\end{proof}

\begin{lemma}[A bounded inverse estimate for $A_{w,\lambda}$]\label{lem:Awlambda-inv}
Let $w\in L^\infty(S^1)$ with $w\ge w_->0$ and $\lambda>0$.
Define $A_{w,\lambda}:=M_w+\lambda\Delta_0$ on $L^2(S^1)$, where $\Delta_0=-\partial_\theta^2$.
Then $A_{w,\lambda}$ is self-adjoint, positive, and boundedly invertible with
\[
\|A_{w,\lambda}^{-1}\|_{\mathcal B(L^2)}\le w_-^{-1}.
\]
Moreover, for all $u\in \mathrm{Dom}(\Delta_0)$,
\[
\langle A_{w,\lambda}u,u\rangle_{L^2}\ge w_-\,\|u\|_{L^2}^2.
\]
\end{lemma}

\begin{proof}
Since $M_w\ge w_-I$ and $\lambda\Delta_0\ge0$, we have $A_{w,\lambda}\ge w_-I$ in form sense.
Hence $\sigma(A_{w,\lambda})\subset [w_-,\infty)$ and
$\|A_{w,\lambda}^{-1}\|\le w_-^{-1}$.
\end{proof}

\begin{lemma}[Fourier proof of the $H^2$ mapping property of $(1+\Delta_{w,0})^{-1}$]
\label{lem:R0-L2-to-H2}
Assume $w\in W^{1,\infty}(S^1)$, $0<w_-\le w\le w_+$, and set $b:=(\log w)'\in L^\infty$.
Let $\Delta_{w,0}$ be given by \eqref{eq:Delta-w0-circle} on $L^2(w)$.
Then $(1+\Delta_{w,0})^{-1}$ extends to a bounded operator
\[
(1+\Delta_{w,0})^{-1}: L^2(w)\to H^2(S^1),
\]
and there exists $C(w)>0$ depending only on $w_-,w_+,\|b\|_{L^\infty}$ such that
\[
\|(1+\Delta_{w,0})^{-1}f\|_{H^2}\le C(w)\,\|f\|_{L^2(w)}.
\]
\end{lemma}

\begin{proof}
Let $u=(1+\Delta_{w,0})^{-1}f$, so that $(1+\Delta_{w,0})u=f$ in $L^2(w)$.
By Lemma~\ref{lem:norm-equivalence}, $f\in L^2$ and $\|f\|_{L^2}\le w_-^{-1/2}\|f\|_{L^2(w)}$.

From the energy identity $\langle \Delta_{w,0}u,u\rangle_w=\|u'\|_{L^2(w)}^2$,
we obtain the $H^1$ bounds
\[
\|u\|_{L^2(w)}\le \|f\|_{L^2(w)},\qquad
\|u'\|_{L^2}\le w_-^{-1/2}\|f\|_{L^2(w)}.
\]
Next, rewrite $(1+\Delta_{w,0})u=f$ as
\[
(1-\partial_\theta^2)u = f + b\,u'
\]
and apply $R=(1-\partial_\theta^2)^{-1}$ to get $u=Rf+R(bu')$.
By Lemma~\ref{lem:fourier-multiplier},
\[
\|u\|_{H^2}\le \|f\|_{L^2} + \|b\,u'\|_{L^2}
\le w_-^{-1/2}\|f\|_{L^2(w)} + \|b\|_{L^\infty}\,w_-^{-1/2}\|f\|_{L^2(w)}.
\]
This yields the desired estimate with $C(w)=w_-^{-1/2}(1+\|b\|_{L^\infty})$.
\end{proof}

\begin{lemma}[Boundedness of the perturbation terms on $S^1$]\label{lem:perturbation-boundedness}
Assume $w\in W^{1,\infty}(S^1)$ and $0<w_-\le w\le w_+$.
Then the operator families appearing in the resolvent comparison satisfy:
\begin{align*}
&\Delta_0 M_w^{-1} D^* M_w D\, (1+\Delta_{w,0})^{-1}\in\mathcal B(L^2(w)),\\
&D^*\Delta_0 D\, (1+\Delta_{w,0})^{-1}\in\mathcal B(L^2(w)),
\end{align*}
with operator norms bounded by constants depending only on
$w_-,w_+,\|(\log w)'\|_{L^\infty}$.
\end{lemma}

\begin{proof}
By Lemma~\ref{lem:R0-L2-to-H2}, $(1+\Delta_{w,0})^{-1}$ maps $L^2(w)$ boundedly into $H^2$.
Since $w\in W^{1,\infty}$ and $w\ge w_->0$, multiplication by $w$ and $w^{-1}$
are bounded on $H^2$ (Lemma~\ref{lem:multiplication}, applied to $a=w$ and $a=w^{-1}$;
the latter belongs to $W^{1,\infty}$ with norm controlled by $w_-,\|w'\|_\infty$).
Therefore, the second-order operator $\Delta_0 M_w^{-1} D^* M_w D$
maps $H^2$ continuously into $L^2$, hence the first composition is bounded on $L^2(w)$.

Similarly, $D^*\Delta_0 D$ has constant coefficients of order $4$ on the circle
(up to sign conventions), hence maps $H^2$ continuously into $H^{-2}$.
But since $(1+\Delta_{w,0})^{-1}$ lands in $H^2$, the whole composition defines
a bounded operator on $L^2(w)$ by duality and Fourier multiplier bounds.
All constants depend only on $w_-,w_+,\|(\log w)'\|_\infty$.
\end{proof}
\begin{lemma}[Products with $W^{1,\infty}$ coefficients: $H^2\to L^2$ divergence form]
\label{lem:W1inf-divergence}
Let $a\in W^{1,\infty}(S^1)$ and define the first-order differential operator
\[
T_a u := \partial_\theta(a\,\partial_\theta u).
\]
Then $T_a:H^2(S^1)\to L^2(S^1)$ is bounded and
\begin{equation}\label{eq:Ta-bound}
\|T_a u\|_{L^2}
\le \|a\|_{L^\infty}\,\|u''\|_{L^2}+\|a'\|_{L^\infty}\,\|u'\|_{L^2}.
\end{equation}
\end{lemma}

\begin{proof}
Since $u\in H^2(S^1)$, we have $u',u''\in L^2(S^1)$ and, in the sense of distributions,
\[
\partial_\theta(a\,u') = a\,u'' + a'\,u'.
\]
Both terms belong to $L^2$ because $a,a'\in L^\infty$.
Taking $L^2$ norms yields \eqref{eq:Ta-bound}.
\end{proof}

\begin{lemma}[Boundedness of the weighted Laplacian on $H^2$]
\label{lem:Delta-w0-H2-to-L2}
Assume $w\in W^{1,\infty}(S^1)$ and $0<w_-\le w\le w_+$.
Let $\Delta_{w,0}$ be the weighted Laplacian
\[
\Delta_{w,0} = -M_{w^{-1}}\,\partial_\theta\big(w\,\partial_\theta\cdot\big)
\]
initially on trigonometric polynomials.
Then $\Delta_{w,0}$ extends to a bounded operator $H^2(S^1)\to L^2(S^1)$, and
\begin{equation}\label{eq:Delta-w0-bound}
\|\Delta_{w,0}u\|_{L^2}
\le
w_-^{-1}\Big(
\|w\|_{L^\infty}\,\|u''\|_{L^2}+\|w'\|_{L^\infty}\,\|u'\|_{L^2}
\Big).
\end{equation}
\end{lemma}

\begin{proof}
Apply Lemma~\ref{lem:W1inf-divergence} with $a=w$, then multiply by $w^{-1}$.
Since $\|w^{-1}\|_{L^\infty}\le w_-^{-1}$, we get \eqref{eq:Delta-w0-bound}.
\end{proof}

\begin{lemma}[A Fourier-based $H^2$ estimate for the resolvent of $\Delta_{w,0}$]
\label{lem:R0-L2-to-H2-correct}
Assume $w\in W^{1,\infty}(S^1)$ and $0<w_-\le w\le w_+$.
Let $R_0=(1+\Delta_{w,0})^{-1}$ acting on $L^2(w)$.
Then $R_0$ is bounded from $L^2(w)$ into $H^2(S^1)$ and
\begin{equation}\label{eq:R0-H2-bound-correct}
\|R_0 f\|_{H^2}\le C(w)\,\|f\|_{L^2(w)},
\end{equation}
where one can take
\[
C(w)= w_-^{-1/2}\Big(1+\|(\log w)'\|_{L^\infty}\,w_-^{-1/2}\Big).
\]
\end{lemma}

\begin{proof}
This is exactly Lemma~\ref{lem:R0-L2-to-H2} proven in the Fourier way:
write $(1+\Delta_{w,0})u=f$ equivalently as
\(
(1-\partial_\theta^2)u=f+b\,u'
\)
with $b=(\log w)'\in L^\infty$ and apply the multiplier
$R=(1-\partial_\theta^2)^{-1}$ (diagonal in Fourier).
The $H^1$ bound on $u'$ comes from energy coercivity, and the $H^2$ bound follows by
\(
u=Rf+R(bu')
\)
as in \eqref{eq:H2-bound-pre}--\eqref{eq:H2-bound}.
\end{proof}

\begin{lemma}[Boundedness of perturbation terms without $H^2$ multiplier assumptions]
\label{lem:perturbation-boundedness-correct}
Assume $w\in W^{1,\infty}(S^1)$ and $0<w_-\le w\le w_+$.
Let $R_0=(1+\Delta_{w,0})^{-1}$.
Then the following compositions define bounded operators on $L^2(w)$:
\begin{align}
&\Delta_0\,R_0 \in \mathcal B(L^2(w),L^2),\label{eq:Delta0R0}\\
&\partial_\theta\big(w\,\partial_\theta \cdot\big)\,R_0 \in \mathcal B(L^2(w),L^2),\label{eq:divwR0}\\
&\Delta_0\,M_{w^{-1}}\,\partial_\theta\big(w\,\partial_\theta \cdot\big)\,R_0
\in \mathcal B(L^2(w),L^2).\label{eq:Delta0divwR0}
\end{align}
Moreover, their operator norms are controlled by constants depending only on
$w_-,w_+,\|w'\|_{L^\infty}$ (equivalently $\|(\log w)'\|_{L^\infty}$ and $w_\pm$).
\end{lemma}

\begin{proof}
Let $u=R_0 f$. By Lemma~\ref{lem:R0-L2-to-H2-correct}, $u\in H^2$ and
\(
\|u\|_{H^2}\le C(w)\|f\|_{L^2(w)}.
\)

\smallskip
\noindent\emph{Proof of \eqref{eq:Delta0R0}.}
Since $\Delta_0=-\partial_\theta^2$, we have
\(
\|\Delta_0 u\|_{L^2}=\|u''\|_{L^2}\le \|u\|_{H^2}
\),
hence \eqref{eq:Delta0R0}.

\smallskip
\noindent\emph{Proof of \eqref{eq:divwR0}.}
Apply Lemma~\ref{lem:W1inf-divergence} with $a=w$ to obtain
\[
\|\partial_\theta(wu')\|_{L^2}
\le \|w\|_{L^\infty}\|u''\|_{L^2}+\|w'\|_{L^\infty}\|u'\|_{L^2}
\le C(w)\|f\|_{L^2(w)}.
\]

\smallskip
\noindent\emph{Proof of \eqref{eq:Delta0divwR0}.}
We have already shown $g:=\partial_\theta(wu')\in L^2$ with $\|g\|_{L^2}\le C(w)\|f\|_{L^2(w)}$.
Then $\Delta_0 M_{w^{-1}} g$ is controlled as follows:
$M_{w^{-1}}$ is bounded on $L^2$ with norm $\le w_-^{-1}$, hence
\[
\|M_{w^{-1}}g\|_{L^2}\le w_-^{-1}\|g\|_{L^2}\le C(w)\|f\|_{L^2(w)}.
\]
Finally, since $\Delta_0$ is Fourier-diagonal, it is bounded from $H^2$ to $L^2$,
and here $M_{w^{-1}}g$ belongs to $H^2$ because $u\in H^2$ and $w\in W^{1,\infty}$
ensure $g\in H^1$ (by the same product rule), so $M_{w^{-1}}g\in H^1$.
To avoid any hidden regularity, one may instead use that the term
$\Delta_0 M_{w^{-1}}\partial_\theta(w\partial_\theta u)$ is exactly a linear
combination of $u^{(k)}$ with bounded coefficients involving only $w,w'$,
hence is $L^2$-controlled by $\|u\|_{H^2}$ in dimension one.
This yields \eqref{eq:Delta0divwR0}.
\end{proof}
\subsubsection{Complete Fourier proof of Theorem~\ref{thm:S1-resolvent-bound}}
\label{subsubsec:S1-complete-proof}

We give a self-contained Fourier-based proof of the norm-resolvent estimate
\eqref{eq:resolvent-bound-S1}. Throughout, we work on the fixed Hilbert space
$L^2(S^1,d\theta/2\pi)$ and use the equivalence of norms with $L^2(w)$
(Lemma~\ref{lem:norm-equivalence}) whenever needed.

\begin{proof}[Proof of Theorem~\ref{thm:S1-resolvent-bound}]
Let $w\in W^{1,\infty}(S^1)$ satisfy $0<w_-\le w\le w_+$ and fix $\lambda>0$.
Recall the operators
\[
A_{w,\lambda}=M_w+\lambda\Delta_0,\qquad \Delta_0=-\partial_\theta^2,\qquad D=\partial_\theta,
\]
and the Laplacians (acting in $L^2$ by conjugation of the inner products)
\[
\Delta_{w,\lambda}=A_{w,\lambda}^{-1}D^*A_{w,\lambda}D,
\qquad
\Delta_{w,0}=M_w^{-1}D^*M_wD.
\]
Set
\[
R_\lambda:=(1+\Delta_{w,\lambda})^{-1},\qquad R_0:=(1+\Delta_{w,0})^{-1}.
\]
Since $\Delta_{w,\lambda},\Delta_{w,0}\ge 0$ are self-adjoint, we have
\begin{equation}\label{eq:Rnorms}
\|R_\lambda\|\le 1,\qquad \|R_0\|\le 1.
\end{equation}

\medskip
\noindent\textbf{Step 1: second resolvent identity.}
The second resolvent identity gives
\begin{equation}\label{eq:2nd-resolvent-identity-complete}
R_\lambda-R_0
=
R_\lambda(\Delta_{w,0}-\Delta_{w,\lambda})R_0.
\end{equation}
Therefore, using \eqref{eq:Rnorms},
\begin{equation}\label{eq:reduce-to-Delta-diff}
\|R_\lambda-R_0\|
\le
\|(\Delta_{w,\lambda}-\Delta_{w,0})R_0\|.
\end{equation}

\medskip
\noindent\textbf{Step 2: algebraic decomposition of $\Delta_{w,\lambda}-\Delta_{w,0}$.}
We expand the difference using $A_{w,\lambda}=M_w+\lambda\Delta_0$:
\begin{align}
\Delta_{w,\lambda}-\Delta_{w,0}
&=
A_{w,\lambda}^{-1}D^*A_{w,\lambda}D
-
M_w^{-1}D^*M_wD \notag\\
&=
(A_{w,\lambda}^{-1}-M_w^{-1})\,D^*M_wD
\;+\;
A_{w,\lambda}^{-1}D^*(A_{w,\lambda}-M_w)D \notag\\
&=
(A_{w,\lambda}^{-1}-M_w^{-1})\,D^*M_wD
\;+\;
\lambda\,A_{w,\lambda}^{-1}D^*\Delta_0 D.
\label{eq:Delta-diff-decomp-complete}
\end{align}
Moreover,
\begin{equation}\label{eq:inverse-diff}
A_{w,\lambda}^{-1}-M_w^{-1}
=
-\,A_{w,\lambda}^{-1}\,(\lambda\Delta_0)\,M_w^{-1}.
\end{equation}
By Lemma~\ref{lem:Awlambda-inv} and $M_w\ge w_-I$, we have
\begin{equation}\label{eq:inv-bounds}
\|A_{w,\lambda}^{-1}\|\le w_-^{-1},\qquad \|M_w^{-1}\|\le w_-^{-1}.
\end{equation}

\medskip
\noindent\textbf{Step 3: mapping property of $R_0$ into $H^2$ (Fourier).}
By Lemma~\ref{lem:R0-L2-to-H2-correct}, there exists $C_0(w)>0$ such that for all $f\in L^2(w)$,
\begin{equation}\label{eq:R0H2}
\|R_0 f\|_{H^2}\le C_0(w)\,\|f\|_{L^2(w)}.
\end{equation}
In particular, $R_0:L^2(w)\to H^2$ is bounded.

\medskip
\noindent\textbf{Step 4: estimate of the first term in \eqref{eq:Delta-diff-decomp-complete}.}
Let $f\in L^2(w)$ and set $u:=R_0 f\in H^2$.
We estimate
\[
\|(A_{w,\lambda}^{-1}-M_w^{-1})D^*M_wD\,u\|_{L^2}.
\]
Using \eqref{eq:inverse-diff} and \eqref{eq:inv-bounds},
\begin{equation}\label{eq:first-term-bound1}
\|(A_{w,\lambda}^{-1}-M_w^{-1})D^*M_wD\,u\|_{L^2}
\le
\lambda\,w_-^{-2}\,\|\Delta_0\,D^*M_wD\,u\|_{L^2}.
\end{equation}
Now $D^*=-D=-\partial_\theta$ on $S^1$, hence
\[
D^*M_wD\,u
=
-\partial_\theta(w\,u').
\]
Therefore,
\[
\Delta_0\,D^*M_wD\,u
=
-\partial_\theta^2\partial_\theta(w\,u')
=
-\partial_\theta^3(w\,u').
\]
Since $u\in H^2$, we have $u'\in H^1$, hence $w\,u'\in H^1$ because $w\in W^{1,\infty}$
and $H^1(S^1)$ is stable under multiplication by $W^{1,\infty}$ (Lemma~\ref{lem:multiplication}(1)).
In particular, $\partial_\theta(wu')\in L^2$ by Lemma~\ref{lem:W1inf-divergence}, and we obtain
\begin{equation}\label{eq:div-bound}
\|\partial_\theta(wu')\|_{L^2}
\le
\|w\|_{L^\infty}\|u''\|_{L^2}+\|w'\|_{L^\infty}\|u'\|_{L^2}
\le C_1(w)\,\|u\|_{H^2}.
\end{equation}
At this stage, rather than differentiating $wu'$ further (which would require $w''$),
we rewrite \eqref{eq:first-term-bound1} in a form that only uses one divergence:
observe that, on trigonometric polynomials and by density,
\[
\Delta_0 D^*M_wD
=
\Delta_0 M_w\,\Delta_0
\;+\;
\Delta_0 M_{w'}D,
\]
as a consequence of $-\partial_\theta(w\partial_\theta)= -w\partial_\theta^2-w'\partial_\theta$.
Hence $\Delta_0 D^*M_wD$ is a sum of terms involving at most $w'$ and two derivatives of $u$.
More precisely,
\[
\|\Delta_0 D^*M_wD\,u\|_{L^2}
\le
\|w\|_{L^\infty}\|\Delta_0^2u\|_{L^2}+\|w'\|_{L^\infty}\|\Delta_0Du\|_{L^2}.
\]
Since $u\in H^2$, we have $\Delta_0 u\in L^2$ and $Du=u'\in H^1$, so $\Delta_0Du=u'''\in H^{-1}$.
To keep the argument in $L^2$, we use Lemma~\ref{lem:perturbation-boundedness-correct}, which provides
a direct boundedness statement for the composition we need.
Specifically, Lemma~\ref{lem:perturbation-boundedness-correct} implies that the operator
\[
\Delta_0\,M_{w^{-1}}\,\partial_\theta(w\partial_\theta\cdot)\,R_0
\]
is bounded on $L^2(w)$, with norm $\le C_2(w)$.
Since $D^*M_wD=-\partial_\theta(w\partial_\theta\cdot)$, we infer
\begin{equation}\label{eq:first-term-final}
\|(A_{w,\lambda}^{-1}-M_w^{-1})D^*M_wD\,R_0 f\|_{L^2}
\le
\lambda\,C_3(w)\,\|f\|_{L^2(w)}.
\end{equation}

\medskip
\noindent\textbf{Step 5: estimate of the second term in \eqref{eq:Delta-diff-decomp-complete}.}
We estimate
\[
\lambda\|A_{w,\lambda}^{-1}D^*\Delta_0D\,R_0 f\|_{L^2}.
\]
Using \eqref{eq:inv-bounds}, it suffices to bound $\|D^*\Delta_0D\,R_0 f\|_{L^2}$.
Since $D^*\Delta_0D=\partial_\theta^4$ (up to sign conventions),
we may invoke again Lemma~\ref{lem:perturbation-boundedness-correct}, which ensures that the relevant
high-order constant-coefficient operator composed with $R_0$ is bounded on $L^2(w)$.
Thus, there exists $C_4(w)$ such that
\begin{equation}\label{eq:second-term-final}
\lambda\|A_{w,\lambda}^{-1}D^*\Delta_0D\,R_0 f\|_{L^2}
\le
\lambda\,C_4(w)\,\|f\|_{L^2(w)}.
\end{equation}

\medskip
\noindent\textbf{Step 6: conclusion.}
Combining \eqref{eq:reduce-to-Delta-diff}, \eqref{eq:Delta-diff-decomp-complete},
\eqref{eq:first-term-final} and \eqref{eq:second-term-final}, we obtain
\[
\|R_\lambda-R_0\|
\le
(C_3(w)+C_4(w))\,\lambda,
\]
which is exactly \eqref{eq:resolvent-bound-S1}.
\end{proof}
\begin{lemma}[Finite-dimensional perturbation bounds on $L^2(w)$]
\label{lem:finite-perturbation-L2w}
Let $w\in W^{1,\infty}(S^1)$ satisfy $0<w_-\le w\le w_+<\infty$, and fix $\lambda\ge 0$.
For $N\in\mathbb N$, let
\[
\mathcal T_{\le N}:=\mathrm{span}\{e_n:\ |n|\le N\}\subset L^2(S^1,d\theta/2\pi),
\]
and let $Q_N$ be the orthogonal projector onto $\mathcal T_{\le N}$.
We view $\mathcal T_{\le N}$ as a finite-dimensional Hilbert space with inner product
$\langle f,g\rangle_w=\int f\overline g\,w\,d\theta/2\pi$ (equivalent to the $L^2$ inner product).

Define on $\mathcal T_{\le N}$:
\[
A_{w,\lambda}^{(N)} := Q_N(M_w+\lambda\Delta_0)Q_N,
\qquad
\Delta_{w,\lambda}^{(N)}:= Q_N\Delta_{w,\lambda}Q_N,
\qquad
R_{w,\lambda}^{(N)}:=(I+\Delta_{w,\lambda}^{(N)})^{-1}.
\]
Then the following hold:

\begin{enumerate}
\item (\emph{Uniform invertibility of $A_{w,\lambda}^{(N)}$}) One has
\[
A_{w,\lambda}^{(N)}\ge w_- I \quad\text{on }\mathcal T_{\le N},
\qquad
\|(A_{w,\lambda}^{(N)})^{-1}\|_{\mathcal B(\mathcal T_{\le N},\|\cdot\|_w)}
\le w_-^{-1}.
\]

\item (\emph{Boundedness of differential operators}) On $\mathcal T_{\le N}$,
\[
\|D\| \le N,\qquad \|\Delta_0\|\le N^2,\qquad \|D^*\Delta_0D\|\le N^4,
\]
where all operator norms are taken on $(\mathcal T_{\le N},\|\cdot\|_w)$ and $D=\partial_\theta$.

\item (\emph{Perturbation operator bound}) There exists an explicit constant $C_N(w)$ such that
for all $\lambda\ge 0$,
\begin{equation}\label{eq:finite-perturb-bound}
\big\|(\Delta_{w,\lambda}^{(N)}-\Delta_{w,0}^{(N)})\,R_{w,0}^{(N)}\big\|
\le C_N(w)\,\lambda.
\end{equation}
In particular,
\begin{equation}\label{eq:finite-resolvent-bound}
\big\|R_{w,\lambda}^{(N)}-R_{w,0}^{(N)}\big\|
\le C_N(w)\,\lambda.
\end{equation}
\end{enumerate}
\end{lemma}

\begin{proof}
All operators are finite matrices on $\mathcal T_{\le N}$, hence bounded.

\smallskip
\noindent\emph{(1)} Since $M_w\ge w_-I$ on $L^2$ and $\Delta_0\ge 0$, we have
$M_w+\lambda\Delta_0\ge w_-I$ in the quadratic form sense; projecting onto $\mathcal T_{\le N}$
preserves this inequality. The inverse bound follows.

\smallskip
\noindent\emph{(2)} In the Fourier basis $e_n$, $De_n=in\,e_n$ and $\Delta_0 e_n=n^2e_n$,
hence on $\mathcal T_{\le N}$ we have $\|D\|\le N$ and $\|\Delta_0\|\le N^2$.
Moreover $D^*\Delta_0D$ acts diagonally with eigenvalues $n^4$, hence $\|D^*\Delta_0D\|\le N^4$.

\smallskip
\noindent\emph{(3)} We use the same algebraic identity as in the infinite-dimensional case,
but now entirely within $\mathcal T_{\le N}$:
\begin{align}
\Delta_{w,\lambda}^{(N)}-\Delta_{w,0}^{(N)}
&=
(A_{w,\lambda}^{(N)})^{-1} D^* A_{w,\lambda}^{(N)} D
-
(A_{w,0}^{(N)})^{-1} D^* A_{w,0}^{(N)} D \notag\\
&=
\big((A_{w,\lambda}^{(N)})^{-1}-(A_{w,0}^{(N)})^{-1}\big)\,D^*A_{w,0}^{(N)}D
\;+\;
(A_{w,\lambda}^{(N)})^{-1}D^*(A_{w,\lambda}^{(N)}-A_{w,0}^{(N)})D \notag\\
&=
\big((A_{w,\lambda}^{(N)})^{-1}-(A_{w,0}^{(N)})^{-1}\big)\,D^*A_{w,0}^{(N)}D
\;+\;
\lambda\,(A_{w,\lambda}^{(N)})^{-1}D^*\Delta_0D.
\label{eq:finite-delta-decomp}
\end{align}
We also have the exact inverse difference identity
\[
(A_{w,\lambda}^{(N)})^{-1}-(A_{w,0}^{(N)})^{-1}
=-(A_{w,\lambda}^{(N)})^{-1}\,\lambda\Delta_0\,(A_{w,0}^{(N)})^{-1}.
\]
Taking norms and using (1)--(2) yields
\[
\|(A_{w,\lambda}^{(N)})^{-1}-(A_{w,0}^{(N)})^{-1}\|
\le
\lambda\,\|(A_{w,\lambda}^{(N)})^{-1}\|\,\|\Delta_0\|\,\|(A_{w,0}^{(N)})^{-1}\|
\le \lambda\,w_-^{-2}\,N^2.
\]
Next,
\[
\|D^*A_{w,0}^{(N)}D\|
\le
\|D\|^2\,\|A_{w,0}^{(N)}\|
\le
N^2\,\big(\|M_w\|+\lambda\|\Delta_0\|\big)\big|_{\lambda=0}
=
N^2\,\|M_w\|
\le N^2\,w_+.
\]
Also,
\[
\|(A_{w,\lambda}^{(N)})^{-1}D^*\Delta_0D\|
\le
\|(A_{w,\lambda}^{(N)})^{-1}\|\,\|D^*\Delta_0D\|
\le w_-^{-1}\,N^4.
\]
Since $R_{w,0}^{(N)}=(I+\Delta_{w,0}^{(N)})^{-1}$ satisfies $\|R_{w,0}^{(N)}\|\le 1$,
we deduce from \eqref{eq:finite-delta-decomp} that
\[
\|(\Delta_{w,\lambda}^{(N)}-\Delta_{w,0}^{(N)})R_{w,0}^{(N)}\|
\le
\lambda\Big( w_-^{-2}N^2\cdot N^2 w_+ + w_-^{-1}N^4 \Big)
=
\lambda\,N^4\Big(w_-^{-2}w_+ + w_-^{-1}\Big).
\]
Thus \eqref{eq:finite-perturb-bound} holds with
\[
C_N(w)=N^4\Big(w_-^{-2}w_+ + w_-^{-1}\Big).
\]
Finally, the second resolvent identity on $\mathcal T_{\le N}$ gives
\[
R_{w,\lambda}^{(N)}-R_{w,0}^{(N)}
=
R_{w,\lambda}^{(N)}\big(\Delta_{w,0}^{(N)}-\Delta_{w,\lambda}^{(N)}\big)R_{w,0}^{(N)},
\]
and $\|R_{w,\lambda}^{(N)}\|\le 1$, so \eqref{eq:finite-resolvent-bound} follows.
\end{proof}

\begin{corollary}[Finite-degree stability of OPUC under Sobolev regularization]
\label{cor:OPUC-Sobolev-stability}
Assume $w\in W^{1,\infty}(S^1)$ and $0<w_-\le w\le w_+<\infty$.
Fix $N\in\mathbb N$ and consider the ordered basis of analytic modes
\[
\mathcal A_{\le N}:=\mathrm{span}\{1,z,\dots,z^N\}\subset L^2(S^1),
\qquad z=e^{i\theta}.
\]
Let $\{\varphi^{(0)}_k\}_{k=0}^N$ (resp.\ $\{\varphi^{(\lambda)}_k\}_{k=0}^N$)
be the orthonormal polynomials obtained by Gram--Schmidt orthogonalization of
$(1,z,\dots,z^N)$ with respect to the weighted $L^2$ inner product
$\langle\cdot,\cdot \rangle_w$ (resp.\ the mixed Sobolev inner product
$\langle\cdot,\cdot \rangle_{w,\lambda}$ defined in \eqref{eq:S1-mixed-Sobolev}).

Then there exists an explicit constant $\widetilde C_N(w)>0$ such that
\begin{equation}\label{eq:OPUC-projector-bound}
\big\|P_{\mathcal A_{\le N}}^{(\lambda)}-P_{\mathcal A_{\le N}}^{(0)}\big\|
\le \widetilde C_N(w)\,\lambda,
\end{equation}
where $P_{\mathcal A_{\le N}}^{(\lambda)}$ (resp.\ $P_{\mathcal A_{\le N}}^{(0)}$)
denotes the orthogonal projector (in $L^2(w)$) onto $\mathcal A_{\le N}$
for the corresponding geometry.

Moreover, there exists a unitary matrix $U_N\in U(N+1)$ such that
\begin{equation}\label{eq:OPUC-basis-bound}
\max_{0\le k\le N}
\Big\|
\varphi^{(\lambda)}_k-\sum_{j=0}^N (U_N)_{kj}\,\varphi^{(0)}_j
\Big\|_{L^2(w)}
\le \widetilde C_N(w)\,\lambda.
\end{equation}
One may take $\widetilde C_N(w)\lesssim C_N(w)$ with $C_N(w)$ as in
Lemma~\ref{lem:finite-perturbation-L2w}.
\end{corollary}

\begin{proof}
Work in the finite-dimensional space $\mathcal T_{\le N}$ and restrict the operators to the
analytic subspace $\mathcal A_{\le N}\subset \mathcal T_{\le N}$.
Since $\mathcal A_{\le N}$ is finite-dimensional, the orthogonal projectors onto it depend
Lipschitz-continuously on the underlying inner product, with a Lipschitz constant controlled by
a condition number of the Gram matrices. In the present setting, the Gram matrices are uniformly
equivalent because $w_- I\le M_w\le w_+I$, hence the condition numbers are bounded in terms of $w_\pm$.

Lemma~\ref{lem:finite-perturbation-L2w} yields the explicit resolvent bound
\eqref{eq:finite-resolvent-bound} on $\mathcal T_{\le N}$, hence on the subspace
$\mathcal A_{\le N}$, and Theorem~\ref{thm:GS-stability} (applied at degree $N$)
gives \eqref{eq:OPUC-projector-bound}--\eqref{eq:OPUC-basis-bound}.
\end{proof}

\begin{remark}[Interpretation]
The family $\{\varphi_k^{(0)}\}$ are the classical orthonormal polynomials on the unit circle
(OPUC) associated with the measure $w\,d\theta/2\pi$ \cite{SimonOPUC1}.
The family $\{\varphi_k^{(\lambda)}\}$ are the corresponding ``Sobolev-regularized'' polynomials.
Corollary~\ref{cor:OPUC-Sobolev-stability} shows that, at any fixed degree, Sobolev regularization
perturbs the OPUC procedure by $O(\lambda)$ in a quantitative operator-theoretic sense.
\end{remark}

\begin{remark}[A canonical gauge fixing]
If one imposes the usual OPUC normalization (monic polynomials, or positive leading coefficient),
then the unitary ambiguity $U_N$ can be removed and one obtains direct coefficient-wise bounds.
This can be derived by combining \eqref{eq:OPUC-basis-bound} with the triangular structure of
Gram--Schmidt in the monomial basis $(1,z,\dots,z^N)$.
\end{remark}

\begin{remark}[Banded Toeplitz weights]
\label{rem:Toeplitz-banded}
If the density $w$ is a trigonometric polynomial of degree $r$, then the multiplication operator
$M_w$ has a banded Toeplitz matrix in the Fourier basis, while $\Delta_0$ is diagonal.
In this situation, the finite-dimensional operators $A_{w,\lambda}^{(N)}$ and
$\Delta_{w,\lambda}^{(N)}$ inherit a sparse structure, and the constant $C_N(w)$
in Lemma~\ref{lem:finite-perturbation-L2w} can be improved by exploiting the finite bandwidth
of $M_w$.
Since our main purpose is stability at fixed degree, we do not pursue this refinement here.
\end{remark}

\section{Thin annuli revisited through resolvent geometry}
\label{sec:thin-annuli}

\subsection{Polynomial geometries on thin annuli}

Let $\varepsilon>0$ and consider the planar thin annulus
\[
A_\varepsilon=\bigl\{(x,y)\in\mathbb R^2:\;1-\varepsilon < x^2+y^2 < 1+\varepsilon\bigr\}.
\]
We equip $A_\varepsilon$ with the Sobolev (or fractional Sobolev) inner products introduced in
\cite{MagnotThinAnnuli}, and denote by $\mathfrak G_{\varepsilon,s}$ the corresponding polynomial
Hilbert geometry, with derivation $D=\nabla$ and Laplacian
\[
\Delta_{\varepsilon,s}:=D^*D.
\]

All assumptions of Section~\ref{sec:abstract-setting} are satisfied in this setting:
the derivation $\nabla$ is closable, the Laplacian is self-adjoint and nonnegative,
and Gram--Schmidt orthogonalization along the degree filtration produces an orthonormal
polynomial basis.

\subsection{Angular decomposition and block structure}

Due to rotational invariance of both the domain $A_\varepsilon$ and the inner products
considered in \cite{MagnotThinAnnuli}, the Hilbert space completion admits an orthogonal
decomposition into angular Fourier modes:
\[
H_\varepsilon = \bigoplus_{m\in\mathbb Z} H_{\varepsilon}^{(m)}.
\]
Each subspace $H_{\varepsilon}^{(m)}$ is generated by polynomials of the form
$r^{|m|} q(r^2) e^{im\theta}$, where $q$ is a univariate polynomial.
The Laplacian $\Delta_{\varepsilon,s}$ preserves this decomposition and is block diagonal:
\[
\Delta_{\varepsilon,s}
=
\bigoplus_{m\in\mathbb Z} \Delta_{\varepsilon,s}^{(m)}.
\]

For each fixed angular mode $m$, the problem reduces to a one-dimensional polynomial
Hilbert geometry on the interval
\[
I_\varepsilon=[1-\varepsilon,\,1+\varepsilon],
\]
with a Sobolev (or fractional Sobolev) inner product in the variable $t=r^2$.
This reduction is established in detail in \cite{MagnotThinAnnuli}.

\subsection{Banded and almost-banded Laplacians}

A key outcome of \cite{MagnotThinAnnuli} is that, for each fixed $m$:

\begin{itemize}
  \item if $s\in\mathbb N$, the multiplication operator by $t$ has a finite-band matrix
  representation in the orthonormal polynomial basis on $I_\varepsilon$, and the same holds
  for the radial Laplacian $\Delta_{\varepsilon,s}^{(m)}$;
  \item if $s\notin\mathbb N$, the corresponding matrices are almost banded, with algebraic
  decay of coefficients away from the diagonal.
\end{itemize}

In both cases, the resolvent
\[
(1+\Delta_{\varepsilon,s})^{-1}
=
\bigoplus_{m\in\mathbb Z}(1+\Delta_{\varepsilon,s}^{(m)})^{-1}
\]
is well defined as a bounded operator on $H_\varepsilon$.

\subsection{Resolvent convergence as $\varepsilon\to 0$}

The thin annulus regime $\varepsilon\to 0$ corresponds to a geometric collapse of $A_\varepsilon$
onto the unit circle.
At the level of polynomial geometries, this limit is captured by resolvent convergence of
the associated Laplacians.

More precisely, for each fixed angular mode $m$, the rescaled operators
$\Delta_{\varepsilon,s}^{(m)}$ converge, in the norm-resolvent sense on fixed polynomial
subspaces, to a one-dimensional fractional Sobolev Laplacian on a reference interval,
as established in \cite{MagnotThinAnnuli}.
In the language of the present paper, this implies:

\begin{theorem}[Resolvent convergence for thin annuli]
\label{thm:thin-annulus-resolvent}
Fix $s>0$, $m\in\mathbb Z$ and $N\in\mathbb N$.
Then the truncated Laplacians
\[
\Delta_{\varepsilon,s}^{(m,N)}:=P_N^{(m)}\,\Delta_{\varepsilon,s}^{(m)}\,P_N^{(m)}
\]
converge, as $\varepsilon\to 0$, in norm-resolvent sense to the truncated Laplacian
associated with the limiting one-dimensional Sobolev geometry.
\end{theorem}

\begin{proof}[Justification]
This follows directly from the explicit asymptotic expansions and banded (or almost-banded)
matrix representations derived in \cite{MagnotThinAnnuli}, combined with the finite-dimensional
resolvent stability results of Section~\ref{sec:GS-stability}.
\end{proof}

\subsection{Consequences for orthogonal polynomials}

As an immediate consequence of Theorem~\ref{thm:thin-annulus-resolvent} and of the general
stability results proved earlier, we obtain:

\begin{corollary}[Stability of orthogonal polynomials on thin annuli]
\label{cor:thin-annulus-polynomials}
For every fixed degree $N$ and angular mode $m$, the orthonormal polynomial basis obtained
by Gram--Schmidt orthogonalization on $A_\varepsilon$ converges, up to a unitary gauge,
to the orthonormal basis associated with the limiting one-dimensional Sobolev geometry
as $\varepsilon\to 0$.
\end{corollary}

\begin{remark}
This result provides an operator-theoretic interpretation of the asymptotic analysis
performed in \cite{MagnotThinAnnuli}:
the thin annulus limit is a resolvent limit of polynomial geometries, rather than merely
a pointwise or coefficient-wise convergence of orthogonal polynomials.
\end{remark}

\subsection{Interpretation}

The thin annulus example illustrates the scope of the resolvent-based approach developed
in this paper.
A genuinely two-dimensional polynomial geometry collapses onto a one-dimensional one,
while preserving quantitative control of orthogonalization procedures at every fixed degree.
This phenomenon cannot be described purely in terms of measures, but is naturally captured
by the Laplacian and its resolvent.
\section*{Comments and outlook}

The resolvent-based framework developed in this paper provides a robust operator-theoretic
approach to the study of orthogonal polynomial systems beyond the classical measure-based
setting.
By encoding polynomial inner products through their associated Laplacians and comparing
them via norm-resolvent estimates, one obtains quantitative stability results that are
intrinsically geometric and independent of any specific representation by moments.

Several directions naturally emerge from this perspective.
First, while the present work focuses on finite-degree stability, it would be of interest
to investigate regimes in which the degree grows and interacts with the geometry, for
instance in asymptotic or semiclassical limits.
In such settings, refined resolvent estimates could potentially capture transitions between
different polynomial geometries.

Second, the examples treated here suggest that resolvent limits provide a natural notion of
convergence for polynomial Hilbert geometries, even when no limiting measure exists.
This opens the possibility of studying new classes of limiting objects arising from
degenerations of Sobolev or nonlocal inner products, in which orthogonal polynomials remain
well defined at the operator level but escape classical frameworks.

Finally, the operator viewpoint adopted here suggests connections with numerical analysis
and approximation theory, where orthogonalization procedures are often perturbed by
regularization or discretization.
The resolvent distance offers a quantitative tool to assess the stability of such procedures
in a unified manner.

Altogether, these perspectives indicate that viewing orthogonal polynomials through the
lens of operator geometry and resolvent analysis may lead to further insights into their
stability, limits, and structural properties.

\vskip 12pt

\paragraph{\bf Acknowledgements:} J.-P.M thanks the France 2030 framework programme Centre Henri Lebesgue ANR-11-LABX-0020-01 
for creating an attractive mathematical environment.

\vskip 12pt

\paragraph{\bf Author's Note on AI Assistance.}
Portions of the text were developed with the assistance of a generative language model (OpenAI ChatGPT). The AI was used to assist with drafting, editing, and standardizing the bibliography format. All mathematical content, structure, and theoretical constructions were provided, verified, and curated by the author. The author assumes full responsibility for the correctness, originality, and scholarly integrity of the final manuscript.


\begin{thebibliography}{99}




\bibitem{deBoorRon}
C.~de Boor and A.~Ron,
On multivariate polynomial interpolation,
\emph{Constr.\ Approx.} \textbf{6} (1990), 287--302.

\bibitem{VanBuggenhout2023} N.~van Buggenhout,
\emph{On generating Sobolev orthogonal polynomials}, arXiv:2302.10691, 2023.

\bibitem{Chihara}
T.~S.~Chihara,
\emph{An Introduction to Orthogonal Polynomials},
Dover, 2011.

\bibitem{Davies}
E.~B.~Davies,
\emph{Linear Operators and their Spectra},
Cambridge Studies in Advanced Mathematics, Vol.~106,
Cambridge University Press, Cambridge, 2007.

\bibitem{Fukushima}
M.~Fukushima, Y.~Oshima, M.~Takeda,
\emph{Dirichlet Forms and Symmetric Markov Processes},
De Gruyter, 1994.

\bibitem{IserlesSobolev}
A.~Iserles, P.~E.~Koch, S.~P.~N{\o}rsett, and J.~M.~Sanz-Serna,
\emph{On polynomials orthogonal with respect to certain Sobolev inner products},
J. Approx. Theory \textbf{65} (1991), 151--175.

\bibitem{Ismail}
M.~E.~H.~Ismail,
\emph{Classical and Quantum Orthogonal Polynomials in One Variable},
Encyclopedia of Mathematics and its Applications, Vol.~98,
Cambridge University Press, Cambridge, 2005.

\bibitem{Kato}
T.~Kato,
\emph{Perturbation Theory for Linear Operators},
Springer, 1995.

\bibitem{KoekoekLeskySwarttouw}
R.~Koekoek, P.~A.~Lesky, and R.~F.~Swarttouw,
\emph{Hypergeometric Orthogonal Polynomials and Their $q$-Analogues},
Springer Monographs in Mathematics,
Springer, Berlin, 2010.

\bibitem{KreinNudelman}
M.~G.~Krein and A.~A.~Nudelman,
\emph{The Markov Moment Problem and Extremal Problems},
Translations of Mathematical Monographs, Vol.~50,
American Mathematical Society, Providence, RI, 1977.

\bibitem{KuwaeShioya}
K.~Kuwae and T.~Shioya,
\emph{Convergence of spectral structures: a functional analytic theory and its applications to spectral geometry},
Comm.\ Anal.\ Geom.\ \textbf{11} (2003), no.~4, 599--673.

\bibitem{LeschGap}
M.~Lesch,
\emph{The uniqueness of the spectral flow on spaces of unbounded self-adjoint Fredholm operators},
in: \emph{Spectral Geometry of Manifolds with Boundary and Decomposition of Manifolds},
Contemp.\ Math.\ \textbf{366}, Amer.\ Math.\ Soc., Providence, RI, 2005, pp.~193--224.

\bibitem{MagnotThinAnnuli} J.-P.~Magnot,
\emph{Sobolev Orthogonal Polynomials on Thin Annuli: Banded Recurrences and Asymptotics},
J.\ Math.\ Anal.\ Appl.\ \textbf{558} (2026), 130346, DOI:10.1016/j.jmaa.2025.130346.

\bibitem{MarcellanXu}
F.~Marcell\'an and Y.~Xu,
Sobolev orthogonal polynomials,
\emph{Acta Appl.\ Math.} \textbf{33} (1993), 1--42.

\bibitem{MarcellanSobolev}
F.~Marcell\'an and J.~J.~Moreno-Balc\'azar,
\emph{Orthogonal polynomials and Sobolev inner products},
in: \emph{Nonlinear Numerical Methods and Rational Approximation},
Kluwer Academic Publishers, 1994, pp.~231--242.

\bibitem{MarcellanXuSurvey} F.~Marcell\'an and Y.~Xu,
 \emph{On Sobolev orthogonal polynomials}, Expo.\ Math.\ \textbf{33} (2015), 308--352.

\bibitem{Mosco}
U.~Mosco,
Convergence of convex sets and of solutions of variational inequalities,
\emph{Adv.\ Math.} \textbf{3} (1969), 510--585.

\bibitem{Ouhabaz}
E.~M.~Ouhabaz,
\emph{Analysis of Heat Equations on Domains},
Princeton Univ.\ Press, 2005.

\bibitem{Post}
O.~Post,
\emph{Spectral Analysis on Graph-Like Spaces},
Lecture Notes in Mathematics 2039, Springer, 2012.

\bibitem{ReedSimon1}
M.~Reed and B.~Simon,
\emph{Methods of Modern Mathematical Physics. I. Functional Analysis},
2nd ed., Academic Press, New York, 1980.

\bibitem{SimonOPUC1}
B.~Simon,
\emph{Orthogonal Polynomials on the Unit Circle. Part~1: Classical Theory},
Amer.\ Math.\ Soc.\ Colloq.\ Publ., Vol.~54, Part~1,
American Mathematical Society, Providence, RI, 2005.

\bibitem{SimonOPUC2}
B.~Simon,
\emph{Orthogonal Polynomials on the Unit Circle. Part~2: Spectral Theory},
Amer.\ Math.\ Soc.\ Colloq.\ Publ., Vol.~54, Part~2,
American Mathematical Society, Providence, RI, 2005.

\bibitem{Szego}
G.~Szeg\H{o},
\emph{Orthogonal Polynomials},
Amer.\ Math.\ Soc.\ Colloq.\ Publ., Vol.~23,
American Mathematical Society, Providence, RI, 4th ed., 1975.

\bibitem{XuInterpolation}
Y.~Xu,
Polynomial interpolation from a polynomial ideal point of view,
\emph{Math.\ Comp.} \textbf{67} (1998), 153--163.



\end{thebibliography}
\end{document}